\documentclass[12pt]{amsart}
\usepackage[utf8]{inputenc}

\title[Compressed sensing for inverse problems II]{Compressed sensing for inverse problems II: applications to deconvolution,  source recovery, and MRI}
\author[G.\ S.\ Alberti]{Giovanni S.\ Alberti}
\address{MaLGa Center, Department of Mathematics, Department of Excellence 2023-2027, University of Genoa, Via Dodecaneso 35, 16146 Genova, Italy.}
\email{giovanni.alberti@unige.it}

\author[A.\ Felisi]{Alessandro Felisi}
\address{MaLGa Center, Department of Mathematics, Department of Excellence 2023-2027, University of Genoa, Via Dodecaneso 35, 16146 Genova, Italy.}
\email{alessandro.felisi@edu.unige.it}

\author[M.\ Santacesaria]{Matteo Santacesaria}
\address{MaLGa Center, Department of Mathematics, Department of Excellence 2023-2027, University of Genoa, Via Dodecaneso 35, 16146 Genova, Italy.}
\email{matteo.santacesaria@unige.it}

\author[S.\ I.\ Trapasso]{S.\ Ivan Trapasso}
\address{Department of Mathematical Sciences ``G.\ L.\ Lagrange'', Politecnico di Torino, Corso Duca degli Abruzzi, 24, 10129 Torino, Italy}
\email{salvatore.trapasso@polito.it}

\date{\today}

\usepackage{amsmath, amssymb, amsfonts, amsthm, amsopn, cancel}
\usepackage{graphicx,tikz}
\usepackage[all]{xy}
\usepackage{amsmath}
\usepackage{bbm}
\usepackage{multirow, longtable, makecell, caption, array, enumitem}
\usepackage{amssymb}
\usepackage{mathtools}
\mathtoolsset{showonlyrefs}
\usepackage{bookmark}
\usepackage{hyperref}

\allowdisplaybreaks

\calclayout

\theoremstyle{plain}
\newtheorem{theorem}{Theorem}[section]
\newtheorem{lemma}[theorem]{Lemma}

\newtheorem{proposition}[theorem]{Proposition}

\theoremstyle{definition}
\newtheorem{assumption}[theorem]{Assumption}
\theoremstyle{definition}
\newtheorem{definition}[theorem]{Definition}
\theoremstyle{definition}

\theoremstyle{plain}
\newtheorem*{theorem*}{Theorem}
\theoremstyle{definition}
\newtheorem{example}[theorem]{Example}
\theoremstyle{remark}
\newtheorem{remark}[theorem]{Remark}

\def\bC{\mathbb{C}}

\def\bN{\mathbb{N}}

\def\bR{\mathbb{R}}

\def\bZ{\mathbb{Z}}

\def\cB{\mathcal{B}}
\def\cD{\mathcal{D}}
\def\cF{\mathcal{F}}
\def\cH{\mathcal{H}}

\def\cHt{\mathcal{H}_2}

\def\cL{\mathcal{L}}

\def\cR{\mathcal{R}}

\def\dim{N}

\def\lc{\left(}
\def\rc{\right)}

\def\supp{\operatorname{supp}}

\def\*b{*_{\bullet}}

\def\Bd'{B_{\delta'}}

\def\cBd'{\bar{B}_{\delta'}}

\newcommand{\diag}{\mathrm{diag}}
\newcommand{\diff}{\mathrm{d}}
\newcommand{\Div}{\mathrm{div}}

\newcommand{\Span}{\mathrm{span}}

\begin{document}

\begin{abstract}
    This paper extends the sample complexity theory for ill-posed inverse problems developed in a recent work by the authors [\textit{Compressed
sensing for inverse problems and the sample complexity of the sparse Radon transform}, J.\ Eur.\ Math.\ Soc., to appear], which was originally focused on the sparse Radon transform. We demonstrate that the underlying abstract framework, based on infinite-dimensional compressed sensing and generalized sampling techniques, can effectively handle a variety of practical applications. Specifically, we analyze three case studies:  (1) The reconstruction of a sparse signal from a finite number of pointwise blurred samples; (2) The recovery of the (sparse) source term of an elliptic partial differential equation from finite samples of the solution; and (3) A moderately ill-posed variation of the classical sensing problem of recovering a wavelet-sparse signal from finite Fourier samples, motivated by magnetic resonance imaging. For each application, we establish rigorous recovery guarantees by verifying the key theoretical requirements, including quasi-diagonalization and coherence bounds. Our analysis reveals that careful consideration of balancing properties and optimized sampling strategies can lead to improved reconstruction performance. The results provide a unified theoretical foundation for compressed sensing approaches to inverse problems while yielding practical insights for specific applications.
\end{abstract}



\maketitle


\section{Introduction}




The compressed sensing (CS) paradigm has fundamentally transformed signal analysis over the past two decades \cite{CRT,donoho2006compressed,eldar-kutyniok-2012,FR,fornasier-2015}. At its core, CS establishes that accurate signal reconstruction is possible using far fewer measurements than traditional sampling methods require, provided the signal is sparse in some appropriate basis. While this insight has led to numerous practical applications, most classical CS theory is formulated in finite-dimensional settings. However, many important inverse problems in imaging and signal processing are inherently infinite-dimensional, necessitating a more sophisticated theoretical framework \cite{AH1,Poon2014,adcock2016generalized}.



This work extends the sample complexity theory for ill-posed inverse problems that we introduced in \cite{split1}. Our framework bridges the gap between infinite-dimensional models and practical scenarios where only finite measurements are available. It rests on three fundamental pillars:
\begin{itemize}
\item Approximate diagonalization of the forward map by elements of the analysis dictionary.
\item Multiscale coherence bounds connecting measurement operators with the sampling and analysis dictionaries.
\item Lower bounds on the sampling probability density that account for measurement noise.
\end{itemize}



We demonstrate how this general theory extends to diverse inverse problems, particularly those involving practical constraints. A key challenge arises when sampling must be restricted to a finite-dimensional subspace of the forward map's range - for example, in practice, Fourier samples can only be acquired within a finite (though potentially large) bandwidth. To address this limitation, we introduce a \textit{balancing property} that aligns with the principles of generalized sampling theory \cite{adcock2016generalized}. This property provides a quantitative framework for ensuring stable transitions from infinite-dimensional settings to tractable finite-dimensional approximations, while controlling the errors inherent in such truncations. Section \ref{sec-abs-mod} provides a detailed mathematical treatment of these concepts.

We also develop novel strategies for optimizing the sampling procedure itself. Our analysis reveals that the sampling probability density can be carefully tailored to match the structure of the coherence bounds, leading to improved sample complexity results. By incorporating both the balancing property and optimized sampling strategies into the abstract recovery framework of \cite{split1}, we significantly expand its applicability while maintaining theoretical rigor.

We demonstrate the framework's versatility through three distinct applications.
\begin{itemize}
    \item 
    In Section~\ref{sec:deconv}, we tackle the deconvolution problem: recovering a signal $f$ that is sparse in its wavelet representation from pointwise samples of its convolution $f*k$ with a blurring kernel $k$ \cite{xiao2011,li-wang2013,tsagkatakis2014,duval-peyre2015}. Beyond establishing the theoretical framework's applicability, we develop an optimized sampling strategy that exploits the kernel's decay properties at infinity. This approach leads to improved recovery guarantees compared to uniform sampling. We then demonstrate the practical significance of our results by providing explicit recovery bounds for cartoon-like images --- an important class of signals characterized by piecewise smooth regions separated by regular discontinuities.


\item Section~\ref{sec:elliptic} addresses an inverse source problem \cite{isakov2,elbadia-2000,biros-dogan-2008,isakov} involving the inversion of a solution operator for an elliptic partial differential equation. Specifically, we consider a boundary value problem on a bounded domain $\Omega \subset \bR^2$ of class $C^2$ with coefficient $\sigma \in C^1(\overline \Omega)$. Given a solution $w \in H^1(\Omega)$ of
\begin{equation}
    \begin{cases}
        -\Div(\sigma\nabla{w}) = u & \text{in $\Omega$},\\
        w|_{\partial\Omega} = 0 & \text{on $\partial\Omega$},
    \end{cases} 
\end{equation}
our goal is to reconstruct the source term $u \in L^2(\Omega)$ from finite pointwise measurements of $w$. This problem presents unique challenges as the forward map lacks translation-equivariance, distinguishing it from the deconvolution setting. We establish recovery guarantees for wavelet-sparse sources using only noisy point samples of the solution. Notably, these results are achieved without requiring additional balancing properties, demonstrating the framework's adaptability to diverse geometric settings.



\item In Section~\ref{sec-fou-ill}, we examine a fundamental problem in compressed sensing: the reconstruction of wavelet-sparse $L^2$ signals from finite Fourier samples \cite{CRT,LDSP,LDP,alberti2017infinite}. We consider an ill-posed variant motivated by magnetic resonance imaging (MRI), where the signal undergoes modulation by a filter that decays away from the origin. Our analysis demonstrates that the balancing property is essential for establishing recovery guarantees when sampling from an optimized probability distribution.

The design of optimal sampling patterns for MRI has generated extensive research, particularly regarding variable density sampling strategies. The literature broadly divides into two approaches: theoretical studies that establish recovery guarantees within the compressed sensing framework \cite{puy2011,chauffert2013,chauffert2014,AHC,krahmer2014,poon2016,boyer2019}, and data-driven methods that optimize sampling patterns for specific applications \cite{ravishankar2011,seeger2010,knoll2011,gozcu2018,gozcu2019,haldar2018,weiss2020,hemant2020,schoenlieb2020}. Our contribution bridges these perspectives by proving that appropriately designed variable density sampling ensures stable reconstruction even in the presence of ill-posedness, thereby extending classical compressed sensing results to more realistic MRI scenarios.
\end{itemize}

The paper is organized as follows. Section~\ref{sec-abs-mod} presents our abstract framework, including the precise formulation of the balancing property and of sampling optimization strategies. In Section~\ref{sec:deconv}, we analyze the sparse deconvolution problem, establishing recovery guarantees and deriving explicit bounds for cartoon-like images. Section~\ref{sec:elliptic} addresses the inverse source problem for elliptic PDEs, demonstrating our framework's applicability to non-translation-invariant settings. Section~\ref{sec-fou-ill} develops the theory for ill-posed Fourier sampling, with particular emphasis on MRI applications. The proofs of all results are collected in Section~\ref{sec-proofs}. Technical results about wavelets and their properties are gathered in Appendix~\ref{appendix:wavelets}.


\section{Setting and main result}\label{sec-abs-mod}
\subsection{Sparsity}
\label{sec:sparsity}
Let $\bN$ denote the set of non-negative integer numbers.

Let $\Gamma$ be a finite or countable double-index set with elements of the form $(j,n)$, where $j\in\bN$ is an index representing the scale. We will consider also finite subsets of $\Gamma$ of the following form:
\begin{align*}
    \Lambda_j \coloneqq \{(j',n)\in\Gamma\colon\ j'= j\},\quad
    \Lambda_{\leq j} \coloneqq \{(j',n)\in\Gamma\colon\ j'\leq j\}.
\end{align*}
We use the notation $M_j\coloneqq |\Lambda_j|$ and $M_{\leq j}\coloneqq |\Lambda_{\leq j}|$ for the cardinality of these sets, and we always assume that $M_j < +\infty$ for every $j\in\bN$.

We denote by $P_j$ the orthogonal projection on $\ell^2(\Gamma)$ defined by \[ (P_j x)_{j',n} = \begin{cases} x_{j',n} & (j',n) \in \Lambda_j \\ 0 & (j',n) \notin \Lambda_j, \end{cases} \] and define similarly the projection $P_{\leq j}$. The image of $P_j$ is thus $\ell^2_{\Lambda_j}(\Gamma) \coloneqq \mathrm{span} \{ e_{j',n} : (j',n) \in \Lambda_j \},$ $(e_{j',n})_{(j',n) \in \Gamma}$ being the canonical basis of $\ell^2(\Gamma)$. We conveniently identify, with an abuse of notation, $\ell^2_{\Lambda_j}(\Gamma)$ with $\ell^2(\Lambda_j)$ or with $\bC^{M_j}$. We denote the corresponding adjoint map by $\iota_{j}$, that is the canonical embedding $\ell^2_{\Lambda_j}(\Gamma) \to \ell^2(\Gamma)$. We also set $P_{j}^\bot \coloneqq I - P_{j}$, where $I$ is the identity operator. Similar notations and conventions are understood when dealing with $P_{\leq j}$ and $\ell_{\Lambda_{\leq j}}^2(\Gamma)$.

For $0<p\leq 2$, we introduce the set
\begin{align*}
    \ell^p(\Gamma)\coloneqq
    \Big\{ x\in\bC^{\Gamma}\colon\ \|x\|_p \coloneqq
    \Big( \sum_{(j,n)\in\Gamma} |x_{j,n}|^p \Big)^{1/p}<+\infty \Big\}.
\end{align*}
We also define
\begin{align*}
    \|x\|_0 \coloneqq |\supp(x)|, \quad x\in\bC^\Gamma,
\end{align*}
where $\supp(x)\coloneqq\{(j,n)\in\Gamma\colon\ x_{j,n}\neq 0\}$.

We now introduce the concept of \textit{sparsity}. Let $s\in\bN$. The following is defined as the  \textit{error of best $s$-sparse approximation} of $x\in\bC^{\Gamma}$ with respect to the $\ell^p$-norm:
\begin{align*}
    \sigma_s(x)_p \coloneqq
    \inf\{ \|x-y\|_p\colon\ y\in\bC^\Gamma,\ \|y\|_0\leq s \}.
\end{align*}
Equivalently, we have that
\begin{align*}
    \sigma_s(x)_p =
    \inf\{ \|x_{S^c}\|_p\colon\ S\subset\Gamma,\ |S|\leq s \},
\end{align*}
where $x_{S^c}$ denotes the projection of $x$ on the indices corresponding to the set $S^c=\Gamma\setminus S$. If $\tilde{S}\subset\Gamma$ is such that $|\tilde{S}|\leq s$ and $\sigma_s(x)_p=\|x_{\tilde{S}^c}\|_p$, we say that $x_{\tilde{S}}$ is a \textit{best $s$-sparse approximation to $x$} with respect to $\ell^p$. Notice that $x_{\tilde{S}}$ is not unique, in general.

We say that a vector $x\in\bC^{\Gamma}$ is \textit{$s$-sparse} if $\|x\|_0\le s$ or, equivalently, if $\sigma_{s}(x)_p=0$ for some $p$.

\subsection{Setting}
\label{sec:setting}
Let us now introduce the mathematical framework that underpins this work. While the foundation is largely derived from \cite{split1}, we extend that approach by incorporating a more general noise model. This enhanced framework not only accommodates a broader class of measurement uncertainties but also enables us to derive more refined reconstruction estimates with explicit dependence on the noise characteristics.


\subsubsection*{Hilbert spaces} Let $\cH_1,\cH_2$ be complex and separable Hilbert spaces. In our model, $\cH_1$  denotes the space of signals and $\cH_2$ the space of measurements.

\subsubsection*{Dictionary} Let $(\phi_{j,n})_{(j,n)\in\Gamma}$ be an orthonormal basis of $\cH_1$ and let $\Phi\colon\cH_1\rightarrow\ell^2(\Gamma)$ be the corresponding analysis operator, so that $\Phi{u}=\lc \langle u,\phi_{j,n} \rangle \rc_{(j,n)\in\Gamma}$. The synthesis operator is then provided by its adjoint $\Phi^*\colon \ell^2(\Gamma)\to \cH_1$, given by $\Phi^* x = \sum_{j,n} x_{j,n} \phi_{j,n}$.

Our results are mainly applied to the case where $(\phi_{j,n})_{(j,n)\in\Gamma}$ is a wavelet dictionary of $L^2(\Omega)$ for some domain $\Omega\subset\bR^\dim$, and we accordingly interpret $j\in\bN$ as a scale index and $n\in\bZ^\dim\times \{0,1\}^\dim$ as a translation and `wavelet type' parameter --- see Appendix~\ref{appendix:wavelets}, where we provide a brief review of the construction of wavelets and of their properties. Notice that, in this case, the constants $M_j$ are proportional to $2^{\dim j}$; this implies that the factor $j_0$ appearing in the estimates can be seen as a logarithmic factor with respect to $M_{\leq j_0}\asymp 2^{\dim j_0}$.

\subsubsection*{Measurement space} Let $(\cD,\mu)$ be a measure space and let $f_{\nu}\in L^1(\mu)$ be a positive probability density function, i.e., $\|f_{\nu}\|_{L^1}=1$. Let $\diff{\nu}=f_{\nu}\diff{\mu}$ be the corresponding probability measure, from which the random samples for the measurements are drawn. The optimal choice of the probability distribution plays a crucial role in the present work (cf.\ Section \ref{sec:prob_distr} below).

\subsubsection*{Measurement operators and forward map} Let
\[
F_t\colon\cH_1\rightarrow\cH_2,\qquad t\in\cD,
\]
be a collection of bounded linear maps. We additionally assume that the mappings $t\mapsto F_t{u}$ belong to the space $L_{\mu}^2(\cD;\cH_2)$ for every $u\in\cH_1$. We define the forward map \[
F\colon\cH_1\mapsto L_{\mu}^2(\cD;\cH_2),\qquad (Fu)(t)\coloneqq F_t{u},
\]
and suppose that $\|F\|\leq C_F$.

\subsubsection*{Noise and truncation error} Given an unknown signal $u^{\dagger}\in\cH_1$ and $m$ i.i.d.\ samples $t_1,\dots,t_m\sim\nu$, the noisy measurements will be the following:
\begin{align*}
    y_k \coloneqq F_{t_k}{u^\dagger} + \varepsilon_k,\quad k=1,\dots,m,
\end{align*}
with $\varepsilon_k\in\cH_2$.

In practice, we approximate $u^{\dagger}$ with a signal that belongs to a fixed finite-dimensio\-nal subspace of $\cH_1$, that is $\Span(\phi_{j,n})_{(j,n)\in\Lambda_{\leq j_0}}$ --- see the minimization problems in Theorem~\ref{thm:abstract}. For the examples discussed below, this subspace is generated by wavelets up to a maximum resolution level $j_0\in\bN$, which must be chosen a priori. This truncation introduces an approximation error $P_{\leq j_0}^{\perp} x^{\dagger}$, which we analyze in detail in the subsequent sections.

In the present work we consider two alternative assumptions on the noise and on the truncation error.
\begin{itemize} 
\item The first is a uniform bound, given by
\begin{equation}
\label{eq:noise_unif_bound}
    \max_k \|\varepsilon_k\|_{\cH_2} \leq \beta,
\end{equation}
where $\beta\ge0$ is the noise level. In this case, for a fixed $\Lambda_{\leq j_0}\subset\Gamma$, we make this assumption on the truncation error:
\begin{equation}
\label{eq:ass-truncation-error}
    \sup_{t\in\cD} \|F_t \Phi^* P_{\leq j_0}^{\perp}x^{\dagger}\|_{\cH_2}\leq r,\qquad
    \|P_{\leq j_0}^{\perp} x^{\dagger}\|_2\leq r,
\end{equation}
where $x^{\dagger}=\Phi{u^{\dagger}}$ and $r\geq 0$. Notice that, if $\sup_{t\in\cD} \|F_t\| \leq C_F$, then the first condition in \eqref{eq:ass-truncation-error} is satisfied with $\lc \max(C_F,1) \rc r$ in place of $r$.

\item The second is a nonuniform bound, given by
\begin{equation}
\label{eq:noise_nonunif_bound}
    \max_k \|f_{\nu}(t_k)^{-1/2}\varepsilon_k\|_{\cH_2} \leq \beta.
\end{equation}
This weighted condition implies that the noise has to be smaller for the measurements corresponding to the samples $t_k$ that are chosen with a lower probability. In this case, the assumption on the truncation error becomes
\begin{equation}
\label{eq:ass-truncation-error2}
    \sup_{t\in\cD} f_{\nu}(t)^{-1/2} \|F_t \Phi^* P_{\leq j_0}^{\perp}x^{\dagger}\|_{\cH_2}\leq r,\qquad
    \|P_{\leq j_0}^{\perp} x^{\dagger}\|_2\leq r.
\end{equation}
where $x^{\dagger}=\Phi{u^{\dagger}}$ and $r\geq 0$. While the first bound in \eqref{eq:ass-truncation-error2} may appear abstract at first glance, it arises naturally in applications. Indeed, this condition is typically satisfied when the truncation error $P_{\leq j}^{\perp}x^{\dagger}$ exhibits appropriate decay as $j\rightarrow +\infty$, which is guaranteed by standard regularity assumptions on the signal --- see Remark~\ref{rem:illfour-nonunif-trunc} for a detailed example.
\end{itemize}

\subsection{Abstract result} We consider the following assumptions.
\begin{assumption}
\label{ass:quasi-diag}
    The following \textit{quasi-diagonalization property} holds:
    \begin{equation}
    \label{eq:quasi-diag}
        c \sum_{(j,n)\in\Gamma} 2^{-2bj}|x_{j,n}|^2 \leq
        \|F\Phi^* x\|_{L_{\mu}^2(\cD;\cH_2)}^2 \leq
        C \sum_{(j,n)\in\Gamma} 2^{-2bj}|x_{j,n}|^2,\qquad x\in \ell^2(\Gamma),
    \end{equation}
    for some $c,C>0$ and $b\geq 0$.
\end{assumption}

The parameter $b\ge 0$ quantifies the level of smoothing of the forward map $F$ and, consequently, the ill-posedness of the inverse problem.

\begin{assumption}
\label{ass:coherence}
    The following coherence bound is satisfied:
    \begin{align*}
        \|F_t\phi_{j,n}\|_{\cH_2} \leq B \frac{\sqrt{f_{\nu}(t)}}{2^{dj}},\quad t\in\cD,\ (j,n)\in\Gamma,
    \end{align*}
    for some $B\geq 1$ and $0\leq d\leq b$.
\end{assumption}

\begin{assumption}
\label{ass:prob_lower_bound}
    The following bounds are satisfied for some $0<c_{\nu}\leq 1$:
    \begin{align*}
        c_{\nu} \leq f_{\nu} \leq 1.
    \end{align*}
\end{assumption}

The following result generalizes \cite[Theorem~3.11]{split1} in two directions. First, for the uniform noise bound \eqref{eq:noise_unif_bound}, we extend the original theorem by modifying the assumptions on the truncation error. Second, we handle the non-uniform noise case \eqref{eq:noise_nonunif_bound} through a transformation: we apply the uniform noise result to the modified measurement operators $F'_t = f_\nu(t)^{-1/2} F_t$. The complete proof is provided in Section~\ref{sec-proofs}.

\begin{theorem}
\label{thm:abstract}
Consider the setting introduced in Section \ref{sec:setting}. Let Assumptions \ref{ass:quasi-diag} and \ref{ass:coherence} be satisfied, and fix $j_0\in\bN$ and $\gamma\in (0,1)$.

\smallskip\noindent
Consider:
\begin{itemize}
    \item a signal $x^{\dagger}\in\ell^2(\Gamma)$;
    \item i.i.d.\ samples $t_1,\dots,t_m\in\cD$ with $t_i\sim\nu$;
    \item measurements $y_k \coloneqq F_{t_k}\Phi^* x^{\dagger}+\varepsilon_k$, for $k=1,\dots,m$;
    \item a parameter $\zeta\in[0,1]$ and weight matrix $W\coloneqq\diag(2^{bj})_{(j,n)\in\Lambda_{\leq j_0}}$, where $b$ is from Assumption \ref{ass:quasi-diag};
    \item and a sparsity parameter $s\in\bN$ with $2\leq s\leq M_{\leq j_0}$.
\end{itemize}
Define
\begin{equation}
    \tau\coloneqq B^2 \,2^{2(b-d)j_0} \, 2^{2(1-\zeta) bj_0}\, s.
\end{equation}

There exist constants $C_0,C_1,C_2,C_3>0$, depending only on $C_F$ and the quasi-diagonalization bounds in \eqref{eq:quasi-diag}, for which the following results hold.

\medskip\noindent
\textbf{Uniform bound case.} 
Suppose that:
\begin{enumerate}[label=(\roman*)]
    \item Assumption \ref{ass:prob_lower_bound} is satisfied;
    \item $\varepsilon_k$ satisfy \eqref{eq:noise_unif_bound};
    \item and $x^{\dagger}$ satisfies \eqref{eq:ass-truncation-error} for some $r\geq 0$.
\end{enumerate}
Let $\widehat{x}$ be a solution of:
\begin{equation}
\label{eq:thm-min-problem}
    \min_{x\in\ell^2(\Lambda_{\leq j_0})} \|W^{-\zeta}x\|_1 :\quad \frac{1}{m}\sum_{k=1}^m \|F_{t_k}\Phi^*\iota_{\leq j_0}x-y_k\|_{\cHt}^2  \leq \lc \beta + C_3 c_{\nu}^{-1/2} 2^{-bj_0}r \rc^2.
\end{equation}
If
\begin{equation}
\label{eq:sample-complexity}
    m \geq C_0\tau\max\{\log^3{\tau}\log{M_{\le j_0}},\log(1/\gamma)\},
\end{equation}
then, with probability exceeding $1-\gamma$, the following recovery estimates hold:
\begin{align}
\label{eq:rec-est-main-thm-1}
    \|W^{-\zeta}x^{\dagger}-W^{-\zeta}\widehat{x}\|_2 &\leq C_1 \frac{\sigma_s(W^{-\zeta}P_{\leq j_0}x^{\dagger})_1}{\sqrt{s}} + C_2 2^{-\zeta b j_0} c_{\nu}^{-1/2} (2^{b j_0}\beta+c_{\nu}^{-1/2} r),\\
\label{eq:rec-est-main-thm-2}
    \|x^{\dagger}-\widehat{x}\|_2 &\leq C_1 2^{\zeta b j_0} \frac{\sigma_s(W^{-\zeta}P_{\leq j_0}x^{\dagger})_1}{\sqrt{s}} + C_2 c_{\nu}^{-1/2} \lc 2^{b j_0}  \beta+ c_{\nu}^{-1/2} r \rc.
\end{align}

\medskip\noindent
\textbf{Nonuniform bound case.} 
Suppose that:
\begin{enumerate}[label=(\roman*)]
    \item $\varepsilon_k$ satisfy \eqref{eq:noise_nonunif_bound};
    \item and $x^{\dagger}$ satisfies \eqref{eq:ass-truncation-error2} for some $r\geq 0$.
\end{enumerate}
Let $\widehat{x}$ be a solution of:
\begin{equation}
    \min_{x\in\ell^2(\Lambda_{\leq j_0})} \|W^{-\zeta}x\|_1 :\quad \frac{1}{m}\sum_{k=1}^m f_{\nu}(t_k)^{-1}\|F_{t_k}\Phi^*\iota_{\leq j_0}x-y_k\|_{\cHt}^2  \leq \lc \beta + C_3 2^{-bj_0}r \rc^2.
\end{equation}
If $m$ is given by \eqref{eq:sample-complexity}, then, with probability exceeding $1-\gamma$, the following recovery estimates hold:
\begin{align}
    \|W^{-\zeta}x^{\dagger}-W^{-\zeta}\widehat{x}\|_2 &\leq C_1 \frac{\sigma_s(W^{-\zeta}P_{\leq j_0}x^{\dagger})_1}{\sqrt{s}} + C_2 2^{-\zeta b j_0} (2^{b j_0}\beta+ r),\\
    \|x^{\dagger}-\widehat{x}\|_2 &\leq C_1 2^{\zeta b j_0} \frac{\sigma_s(W^{-\zeta}P_{\leq j_0}x^{\dagger})_1}{\sqrt{s}} + C_2 \lc 2^{b j_0}  \beta+ r \rc.
\end{align}
\end{theorem}

\begin{remark}
    We notice for further convenience that Assumption \ref{ass:quasi-diag} can be relaxed in the hypotheses of Theorem \ref{thm:abstract} -- see \cite[Remark 5.12]{split1}. For a fixed scale index $j_0\in\bN$, it suffices that $F$ satisfies the following properties:
    \begin{equation}
    \label{eq:weak-quasi-diag1}
        \|F\Phi^*x\|_2^2 \leq C'\sum_{(j,n)\in\Gamma} 2^{-2bj} |x_{j,n}|^2,\qquad x\in\ell^2(\Gamma),
    \end{equation}
    \begin{equation}
    \label{eq:weak-quasi-diag2}
        c'\sum_{(j,n)\in\Lambda_{\leq j_0}} 2^{-2bj}|x_{j,n}|^2 \leq \|F\Phi^* x\|_2^2,\qquad x\in\ell^2(\Lambda_{\leq j_0})
    \end{equation}
    for some constants $c',C'>0$. In this case, we say that  $F$ satisfies the \textit{weak quasi-diagonalization property} with respect to $(\Phi,j_0,b)$.
\end{remark}

\begin{remark}
\label{rem:ass_coherence_mod}
    It is also possible to weaken Assumption \ref{ass:coherence}. Indeed, the following bound suffices:
    \begin{equation}
    \label{eq:ass_coherence_mod}
        \|F_t\phi_{j,n}\|_{\cH_2} \leq B \frac{\sqrt{f_{\nu}(t)}}{2^{dj}},\quad t\in\cD,\ (j,n)\in\Lambda_{\leq j_0},
    \end{equation}
    for some $B\geq 1$ and $0\leq d\leq b$. The availability of a weaker version of Assumption~\ref{ass:coherence} is crucial when one wishes to extend Theorem \ref{thm:abstract} to cases where a bound with a constant $B\geq 1$ that is uniform with respect to the scale index $j$ is not available. In these cases, $B$  depends on the maximum resolution $j_0$, which in turn implies an additional implicit dependence of the sample complexity on $j_0$.
\end{remark}

\subsection{Balancing property}
\label{sec:balancing}
In many cases of interest, the forward map $F$ used in the reconstruction procedure has the form $P\circ U$, where 
\[
U\colon\cH_1\longrightarrow \cH,
\]
can be viewed as the \textit{natural} forward map of the model, whose codomain is some Hilbert space $\cH$ and $P\in\cL(\cH)$ is a suitable projection operator whose range can be identified with $L^2_{\mu}(\cD;\cH_2)$, namely, the codomain of $F$. We can also have $\cH=L^2_{\mu}(\Sigma;\cH_2)$, where $(\Sigma,\mu)$ is an extension of the measure space $(\cD,\mu)$. In this case, the projection $P$ corresponds to the truncation operator
\begin{align*}
    P\colon u\mapsto Pu(x)=
    \begin{cases}
        u(x) & \text{if $x\in\cD$}\\
        0 & \text{otherwise},
    \end{cases}
\end{align*}
but we stress that this is not to be expected in general --- see, for instance, Example~\ref{ex:radon_trunc}.


We now present two examples that illustrate the decomposition $F=P\circ U$, where $U$ represents the natural forward map of the physical problem and $P$ is the projection operator arising from practical measurement constraints.

\begin{example}[\textbf{Reconstruction of a signal via Fourier samples}]
\label{ex:fou-wav}
Consider the problem of reconstructing a signal $u^{\dagger}\in\cH_1\coloneqq L^2(0,1)$ from samples of its Fourier coefficients $\widehat{u^{\dagger}}(t)\in\cH_2\coloneqq \bC$ for $t\in\bZ$. The natural forward map of the problem coincides with the Fourier transform $U=\cF\colon L^2(0,1)\rightarrow \ell^2(\bZ)$, which is unitary. Note that the codomain is given by $\ell^2(\bZ)=L_{\mu}^2(\Sigma)$, where $\Sigma=\bZ$ and $\mu$ is the counting measure.

However, in practical applications, it is necessary to restrict the space of measurements to a finite number of frequencies. In other words, we need to consider  a random sampling of the frequencies in $[N]_{\pm}\coloneqq\{-N,-N+1,\dots,N-1,N\}$ for some fixed $N\in\bN$. In this case, the forward map that models the problem is given by the \textit{truncated} Fourier transform $F\colon L^2(0,1)\rightarrow \ell^2([N]_{\pm})$, which can be viewed as the composition of $U$ with the projection $P_N\colon\ell^2(\bZ)\rightarrow\ell^2([N]_{\pm})$ on the corresponding frequency bandwidth: $F=P_N\circ U$.

For reconstruction using a wavelet dictionary $(\phi_{j,n})_{(j,n)\in\Gamma}$ of $L^2(0,1)$, the necessity of finite sampling follows from Assumption \ref{ass:coherence}. Indeed, \cite[Theorem 2.1]{JAH} shows that under suitable regularity assumptions on the wavelets:
\begin{equation}
\label{eq:fou-wav-coherence}
    |\cF{\phi_{j,n}}(0)| \leq C,\quad
    |\cF{\phi_{j,n}}(t)| \leq \frac{C}{\sqrt{|t|}}\quad \text{for }t\in\bZ\setminus\{0\},
\end{equation}
where $(j,n)\in\Gamma$ and this bound is essentially optimal. 

Consequently, the minimal probability density $f_{\nu}(t)$ satisfying Assumption~\ref{ass:coherence} with $d=b=0$ (that is, satisfying $|\cF{\phi_{j,n}}(t)| \leq B\sqrt{f_{\nu}(t)}$) must be proportional to $1/|\cdot|$. Since this function is not in $\ell^1(\bZ)$, it cannot be normalized to a probability density with respect to the counting measure, demonstrating why the sampling space must be restricted.
\end{example}

\begin{remark}
    The lack of summability in Example \ref{ex:fou-wav}, hence the need of a suitable truncation, is an unavoidable phenomenon that manifests itself every time one is concerned with reconstructing the sparse coefficients of a signal $u$ with respect to an orthonormal basis $(\phi_i)_{i\in\bN}$ of a Hilbert space $\cH$ by sampling the coefficients $F_tu=\langle u,\psi_t\rangle_\cH$ with respect to another orthonormal basis $(\psi_t)_{t\in\bN}$ \cite{AH1, AHC}. Indeed, consider the (asymptotic) coherence
    \begin{equation}
        \mu_t \coloneqq \sup_{i\in\bN} |\langle \psi_t,\phi_i \rangle_\cH|=\sup_{i\in\bN} |F_t\phi_i|, \quad t \in \bN,
    \end{equation}
    and suppose by contradiction that $\sum_{t\in \bN} \mu_t^2<\infty$. Then, there exists $N\in\bN$ such that $\sum_{t>N}\mu_t^2\leq 1/2$. Let $P_N$ be the orthogonal projection on $\Span(\psi_t)_{t=1}^N$. We have
    \begin{align*}
        1 = \|\phi_i\|_\cH^2 &=
        \sum_{t \ge 1} |\langle \psi_t,\phi_i \rangle|^2  =
        \sum_{t\leq N} |\langle \psi_t,\phi_i \rangle|^2 + 
        \sum_{t>N} |\langle \psi_t,\phi_i \rangle|^2 \leq 
        \sum_{t\leq N} |\langle \psi_t,\phi_i \rangle|^2 + \frac 1 2 ,
    \end{align*}
    which implies that
    \begin{equation}
    \label{eq-rem-coherence}
        \|P_N\phi_i\|_{\cH}^2 \geq \frac 1 2,\quad i\in\bN.
    \end{equation}
    On the other hand, since $\phi_i \to 0$ weakly and $P_N$ is a compact operator, we have $\|P_N\phi_i\|_\cH\rightarrow 0$. We conclude that $\sum_t \mu_t^2 = \infty$, as claimed. 
    
    If a coherence bound like
    \begin{align*}
        |\langle \psi_t,\phi_i \rangle|^2 \leq B^2 f_{\nu}(t),\qquad i,t\in\bN,
    \end{align*}
    were available, $f_{\nu}(t)$ being a probability density with respect to the counting measure, then after taking the supremum over $i$ we would get
    \begin{align*}
        \mu_t^2 \leq B^2 f_{\nu}(t),
    \end{align*}
   hence leading to a contradiction:
    \begin{align*}
        +\infty = \sum_{t\in\bN} \mu_t^2 \leq B^2 \sum_{t\in\bN} f_{\nu}(t) = B^2 < +\infty.
    \end{align*}
     We thus conclude that resorting to the balancing property is necessary in this framework.
\end{remark}

\begin{example}[\textbf{The truncated Radon transform}]
\label{ex:radon_trunc}
Consider the problem of reconstructing a signal $u^\dagger \in L^2(\cB_1)$ from samples of its Radon transform at fixed angles, where $\cB_1\subset\bR^2$ is the open unit ball. For each direction $\theta$, the Radon transform $\cR_\theta\colon L^2(\cB_1)\rightarrow L^2(\bR)$ is defined by
\[
(\cR_\theta u)(s) = \int_{\theta^{\perp}} u(y+s e_\theta) \diff{y},
\]
where $e_{\theta}=(\cos\theta,\sin\theta)$ --- see \cite{natterer}. The samples are drawn from $\lc \cR_{\theta}{u^{\dagger}} \rc_{\theta\in[0,2\pi)}$. 

The Radon transform 
\[
U=\cR\colon L^2(\cB_1)\rightarrow L^2([0,2\pi)\times\bR), \quad (\cR u)(\theta,s)=(\cR_\theta u)(s)
\]
serves as the natural forward map. For details on the Radon inversion with subsampled angles, see \cite{split1}.
The Fourier slice theorem establishes that
\begin{align*}
    \cF_1\cR_{\theta}{u}(\sigma) = \cF_2u(\sigma e_\theta),\qquad \sigma\in\bR,
\end{align*}
where $\cF_1$ and $\cF_2$ are the one-dimensional and two-dimensional Fourier transforms. Thus, $\cR_{\theta}{u}$ coincides spectrally with the restriction of $\cF_2u$ to the radial line spanned by $e_\theta$, making the Radon transform suitable for modeling radial line sampling in MRI.

In applications, measurements are bandwidth-limited: for each $\theta\in[0,2\pi)$, we can only sample $\cF_2u$ on the radial segment 
\[
\ell_{\theta,N}\coloneqq\{\sigma e_\theta\colon\ \sigma\in[-N,N]\}
\] 
for some $N\in[0,+\infty)$. This limitation appears both in MRI sampling and in the sparse angle Radon transform, where $\cR_{\theta}u^{\dagger}$ is only accessible up to finite resolution.

The corresponding forward map becomes 
\[
F=P_{\leq N,\sigma}\cF_{1,s} U,
\]
where $\cF_{1,s}$ is the Fourier transform in the $s$ variable and 
\[
P_{\leq N,\sigma}\colon L^2([0,2\pi)\times\bR)\rightarrow L^2([0,2\pi)\times[-N,N])
\]
projects onto functions supported in $[0,2\pi)\times[-N,N]$.

For the related problem of sampling the Fourier transform along radial lines in the superresolution setting, we refer to \cite{dossal2017}.
\end{example}


Returning to the general framework, we introduce a balancing property --- a condition essential for controlling truncation errors in infinite-dimensional problems. This property, which ensures stability in the transition from infinite to finite-dimensional settings, is analogous to conditions that have appeared previously in the compressed sensing literature \cite{adcock2016generalized}.

\begin{definition}
Given a projection $P\in\cL\lc \cH \rc$, $U\in\cL\lc \cH_1,\cH \rc$, $j_0\in\bN$, $\theta>0$ and $b\geq 0$, we say that $P$ satisfies the \textit{balancing property} with respect to $(U,\Phi,j_0,b,\theta)$ if
\begin{equation}
\label{eq:balancing}
    \|P^\bot U\Phi^*\iota_{\leq j_0}\|_{\cH_1\to \cH}^2 \leq \theta^2 2^{-2bj_0},
\end{equation}
where $P^\perp = I-P$.
\end{definition}

In some situations it is useful to single out a projection $P$ among families like $(P_N)_{N\in\bN}$ in such a way that $U$ satisfies \eqref{eq:balancing} with respect to $P$. This is the case, for instance, of the previous two examples, where the family of projections are given by $P_N$ and $P_{\le N,\sigma}\circ\cF_{1,s}$, respectively.

Suppose that
\begin{align*}
    P_N \xrightarrow{\ \mathrm{str}\ } I
\end{align*}
with respect to the strong operator convergence as $N\rightarrow \infty$ namely $P_N{x}\rightarrow x$ as $N\rightarrow\infty$ for every $x\in\cH_1$. Fix $j_0\in\bN$, $\theta>0$ and $b\geq 0$. Then there exists $\overline{N}=\overline{N}(j_0,\theta,b)$ such that $P_{\overline{N}}$ satisfies the balancing property with respect to $(U,\Phi,j_0,b,\theta)$. Indeed, equicontinuity and pointwise convergence of $P_N^\bot U\Phi^*$ to $0$ implies uniform convergence on the compact set $\{x\in\ell^2(\Lambda_{\leq j_0})\colon\ \|x\|_2\leq 1\}$.

The following result states that if $U$ satisfies the quasi-diagonalization property and $P$ satisfies the balancing property, then the truncated map $F=P\circ U$ satisfies the weak quasi-diagonalization property. The proof is postponed to Section~\ref{sec-proofs}.

\begin{proposition}
\label{prop:balancing}
    Suppose that $U\colon\cH_1\rightarrow \cH$ satisfies the quasi-diagonalization property \eqref{eq:quasi-diag} with respect to $(\Phi,b)$ with constants $c,C>0$ and suppose that $P$ satisfies the balancing property \eqref{eq:balancing} with respect to $(U,\Phi,j_0,b,\theta)$ with $\theta^2\leq c/2$. Then $F=P\circ U$ satisfies the weak quasi-diagonalization property \eqref{eq:weak-quasi-diag1} and \eqref{eq:weak-quasi-diag2} with respect to $(\Phi,j_0,b)$ with constants $C'=C$ and $c'=c/2$, respectively.
\end{proposition}

\subsection{Choice of the probability distribution}
\label{sec:prob_distr}

In many applications, measurement operators satisfy coherence bounds of the form
\begin{equation}
\label{eq:coherence-natural}
    \|F_t{\phi_{j,n}}\|_{\cH_2} \leq \frac{\sqrt{g_{\nu}(t)}}{2^{dj}},\qquad t\in\cD,\ (j,n)\in\Gamma,
\end{equation}
where $g_{\nu}$ is a positive function and $0\leq d\leq b$, with $b$ being the quasi-diagonalization parameter of the natural forward map $U$. This structure appears, for instance, in the reconstruction of wavelet coefficients from Fourier samples (see Example \ref{ex:fou-wav} and \eqref{eq:fou-wav-coherence}).
For fixed $0\leq d\leq b$, the minimal function $g_{\nu}$ satisfying \eqref{eq:coherence-natural} is
\begin{equation}
\label{ass:coherence_gnu}
    g_{\nu}(t) \coloneqq \sup_{(j,n)\in\Gamma} 2^{2dj}\|F_t{\phi_{j,n}}\|_{\cH_2}^2.
\end{equation}
When $g_{\nu}\in L^1(\mu)$, we can normalize it to obtain a probability density $f_{\nu}$ with respect to $\mu$ by setting
\[
    f_{\nu} = C_{\nu}^{-1}g_{\nu}, \quad \text{where} \quad C_{\nu} = \int_{\cD} g_{\nu}\diff{\mu}.
\]
This allows us to rewrite \eqref{eq:coherence-natural} in the form required by Assumption \ref{ass:coherence}:
\begin{equation}
\label{eq:coherence}
    \|F_t{\phi_{j,n}}\|_{\cH_2} \leq B\frac{\sqrt{f_{\nu}(t)}}{2^{dj}},
\end{equation}
with $B=\sqrt{C_{\nu}}$. The recovery guarantees of Theorem \ref{thm:abstract} then apply.

While multiple choices of $B$, $f_{\nu}$, and $d$ may satisfy \eqref{eq:coherence}, optimal sample complexity is achieved by:
\begin{itemize}
    \item choosing $d$ as the largest value in $[0,b]$ for which \eqref{eq:coherence-natural} holds,
    \item given $d$, taking $f_{\nu}$ proportional to the minimal $g_{\nu}$ from \eqref{ass:coherence_gnu}.
\end{itemize}
This optimality is demonstrated below through a classical compressed sensing example.

\begin{example}[\textbf{Optimal sampling for Fourier reconstruction}]
\label{ex:fouwave-continues}
We return to the problem of reconstructing wavelet coefficients of $u^{\dagger}\in L^2(0,1)$ from its Fourier samples $\cF{u}(t)$. As noted in Example \ref{ex:fou-wav}, we have the optimal bound
\begin{equation}
\label{eq:fou-wav-coherence2}
    |\cF{\phi_{j,n}}(t)| \leq \frac{B_0}{\sqrt{|t|}}.
\end{equation}
Since $t\mapsto 1/|t|$ is not in $\ell^1(\bZ\setminus\{0\})$, we require a balancing property. For the projection $P_N\in\cL(\ell^2(\bZ))$ onto $[N]_{\pm}=\{-N,\ldots,N\}$, it is proved in \cite{AHP2, AHC} that under suitable wavelet assumptions, for every $\theta\in(0,1)$ there exists $C_{\theta}$ such that $P_{N_{j_0}}$ satisfies the balancing property with respect to $(\cF,\Phi,j_0,0,\theta)$ for all $j_0\in\bN$, where $N_{j_0}=\lfloor C_{\theta}2^{j_0} \rfloor$.

By Proposition~\ref{prop:balancing}, the restricted Fourier transform 
\[
\cF_{j_0}\colon\Span(\phi_{j,n})_{(j,n)\in\Lambda_{\leq j_0}}\rightarrow\ell^2([N_{j_0}]_{\pm})
\]
satisfies the weak quasi-diagonalization property \eqref{eq:quasi-diag} with $b=0$.

\smallskip\noindent
\textbf{Nonuniform sampling strategy.}
With $b=d=0$, following Section \ref{sec:prob_distr}, we can choose $f_{\nu}$ proportional to $1/|t|$:
\[
f_{\nu}(t) = \begin{cases}
    C_{\nu}^{-1} & t=0\\
    C_{\nu}^{-1}/|t| & t\neq 0
\end{cases}
\]
where
\[
    C_{\nu} = 1+\sum_{t=1}^{N_{j_0}} \frac{2}{|t|} \approx \log{N_{j_0}} +1 \leq C j_0+1.
\]
This satisfies \eqref{eq:ass_coherence_mod} with $B=C_{\nu}^{\frac12}B_0$. Note that $B\to +\infty$ as $j_0\to+\infty$, highlighting the importance of the weaker coherence bound \eqref{eq:ass_coherence_mod}.

Applying Theorem~\ref{thm:abstract} with $d=b=0$ yields stable recovery with sample complexity
\[
    m \gtrsim j_0^2 s,
\]
up to logarithmic factors in $s$ --- matching \cite[Section~3.3]{alberti2017infinite}.

\smallskip\noindent
\textbf{Uniform sampling strategy.}
Alternatively, considering uniform sampling with $f_{\nu} = 1/(2N+1)$, the optimal bound from \eqref{eq:fou-wav-coherence2} gives
\[
    |\cF{\phi_{j,n}}(t)| \leq B_0.
\]
This satisfies \eqref{eq:ass_coherence_mod} with $B=(2N+1)^{\frac12}B_0\approx C2^{j_0/2}$, leading to sample complexity
\[
    m \gtrsim 2^{j_0} j_0^2 s,
\]
up to logarithmic factors in $s$ --- clearly suboptimal compared to nonuniform sampling.
\end{example}

\subsection{Wavelet dictionary}
\label{sec:wavelet-uniform}

For the examples in Sections \ref{sec:deconv}, \ref{sec:elliptic} and \ref{sec-fou-ill}, we consider a wavelet dictionary adapted to signals in $L^2(\Omega)$, where $\Omega\subset\bR^2$ is a fixed domain. We construct this dictionary as follows.

Starting with a dictionary of $r$-regular compactly supported wavelets for $L^2(\bR^\dim)$ (see Appendix~\ref{appendix:wavelets}), we define $(\phi_{j,n})_{(j,n)\in\Gamma}$ as the subfamily whose supports intersect $\Omega$. Here:
\begin{itemize}
    \item $j\in\bN$ is the scale index,
    \item $n\in\bZ^2\times\{0,1\}^2$ encodes translation and wavelet type.
\end{itemize}
Let $\cH_1$ be the $L^2$-closure of the span:
\[
\cH_1 = \overline{\Span(\phi_{j,n})_{j,n\in\Gamma}} \subseteq L^2(\bR^\dim)
\]
and define the total support
\begin{equation}
\label{eq-K}
    K = \overline{\bigcup_{(j,n)\in\Gamma} \supp(\phi_{j,n})}.
\end{equation}
The analysis operator $\Phi\colon\cH_1\rightarrow\ell^2(\Gamma)$ is defined by
\[
\Phi u = (\langle u,\phi_{j,n}\rangle_2)_{j,n}.
\]

\section{Sparse deconvolution}
\label{sec:deconv}
In this section we consider the problem of reconstructing the wavelet coefficients of a signal given pointwise samples of a ``blurred'' version obtained by convolution with a Bessel kernel. The choice of the Bessel potential as kernel is made here in order to have a precise, explicit control on the regularity parameters and the related quantities. Nevertheless, large parts of the following analysis extend to different kernels after minor modifications.

The section is organized as follows. First, in Sec.~\ref{sec:deconv-setting} we introduce the setting for the deconvolution problem. In Sec.~\ref{sec:deconv_balancing} we verify that the balancing and the quasi-diagonalization properties are satisfied, while in Sec.~\ref{sec:deconv-coherencebounds-probdistr} we compute the coherence bounds and define the probability distribution from which samples are taken --- in this case, the sampling distribution is compactly supported in $\bR^2$. This allows us to deduce recovery estimates for the problem in Sec.~\ref{sec:deconv-recovery}. In Sec.~\ref{sec:deconv-alternative-prob} we consider samples taken from a probability distribution supported on $\bR^2$, and state the corresponding result under a stronger assumption on the noise. Finally, in Sec.~\ref{sec:deconv_cartoon_like} we apply our results to obtain explicit bounds for the reconstruction of cartoon-like images. The proofs of all the results are in Section~\ref{sec-proofs}.

\subsection{Setting}
\label{sec:deconv-setting}

We consider a $r$-regular wavelet dictionary $(\phi_{j,n})_{(j,n)}$ adapted to the unit ball $\cB_1\subset\bR^2$ and the corresponding space $\cH_1$, as described in Sec.~\ref{sec:wavelet-uniform}. We consider scalar measurements and set $\cH_2=\bC$. The natural forward map $U\colon\cH_1\rightarrow L^2(\bR^2)$ is given by $U\coloneqq (I-\Delta)^{-b/2}$, with $b>\max(2,r)$.\footnote{The condition $b>1$ is enough to ensure that the Bessel kernel is in $L^2(\bR^2)$. The more restrictive condition $b>2$ is used here for simplicity, in order to explicitly characterize the decay of the kernel, although this condition is not strictly necessary for our setting.} It is shown below that $b$ actually plays the role of the quasi-diagonalization parameter, which motivates resorting to the same symbol. The Bessel potential $(I-\Delta)^{-b/2}$ is defined via the Fourier transform as
\begin{equation}
    (I-\Delta)^{-b/2}u \coloneqq \cF^{-1}\lc (1+|\cdot|^2)^{-b/2} \cF u \rc, \quad u \in L^2(\bR^2),
\end{equation}
where $\mathcal{F}: L^2(\bR^2)\to L^2(\bR^2)$ is the two-dimensional Fourier transform given by
\[
\cF{u}(\xi) = \int_{\bR^2} e^{-2\pi i \xi \cdot x} u(x) \,\diff{x}.
\]
For simplicity of notation, we view $-\Delta$ as the Fourier multiplier with symbol $|\cdot|^2$ rather than the standard $4\pi^2|\cdot|^2$.

The operator $U$ can be expressed as a convolution:
\begin{equation}
   U{u} = (I-\Delta)^{-b/2}u = \kappa_b\ast u,
\end{equation}
where $\kappa_b\coloneqq \cF^{-1}\lc (1+|\cdot|^2)^{-b/2} \rc$ is the convolution kernel.

The following key properties hold:
\begin{enumerate}[label=(\roman*)]
    \item $\|U\|_{L^2\to L^2} = \|(1+|\cdot|^2)^{-b/2}\|_{L^{\infty}}=1$
    \item By definition,
    \begin{equation}
    \label{eq:deconv-quasi-diag}
        \|Uu\|_{L^2} = \|u\|_{H^{-b}}
    \end{equation}
\end{enumerate}

\smallskip\noindent
The second property, combined with the Littlewood-Paley characterization of Sobolev spaces (see Proposition \ref{prop:littlewood-paley} and \cite[Proposition 3.8]{split1}), implies that $U$ satisfies the quasi-diagonalization property with respect to $(\Phi,b)$.

The inverse problem that we consider here consists in reconstructing a signal $u$ given a finite number of pointwise samples of $\kappa_b\ast u$, namely $\lc \kappa_b\ast u(t_1),\dots,\kappa_b\ast u(t_m)\rc$. To model this problem, we introduce measurement operators $F_t\colon \cH_1\rightarrow\bC$, $t \in \bR^2$, defined by $F_t{u} = U{u}(t)$. Notice that $Uu$ is a continuous function, as it is given by the convolution of two functions in $L^2(\bR^2)$. Moreover, by Young's inequality, we have that $|F_t u|\leq \|\kappa_b\|_{L^2}\|u\|_{L^2}$; in other words, $(F_t)_{t\in\bR^2}$ is a family of uniformly bounded operators.

The natural space for this model is $(\Sigma,\mu) = (\bR^2,\diff{x})$, where $\diff{x}$ denotes the Lebesgue measure on $\bR^2$. This choice is natural since, even though the signal $u^\dagger$ to be reconstructed has compact support in $\cB_1$, its convolution with a non-compactly supported kernel will generally have non-compact support in $\bR^2$.

\subsection{Balancing property and quasi-diagonalization}
\label{sec:deconv_balancing}
We now investigate how to perform truncation in order to preserve the energy of the forward map. Let us anticipate that an alternative sampling scheme defined on the whole plane that does not require the balancing property can be found in Section~\ref{sec:deconv-alternative-prob}.

We consider the projections $(P_N)_{N\in\bN}$ on the spaces $L^2(\cB_N)$, with $\cB_N=\{x\in \bR^2:|x|\leq N\}$.
\begin{lemma}
\label{lem:deconv_balancing}
There exists a constant $C>0$, which depends on $b$ and on the wavelet dictionary, such that, if
\begin{align*}
    N \geq C \lc j_0 + \log(1/\theta) + 1 \rc,
\end{align*}
then $P_N$ satisfies the balancing property with respect to $(U,\Phi,j_0,b,\theta)$.
\end{lemma}
To lighten the notation, we can fix $\theta=\sqrt{c/2}$, where $c$ is the lower quasi-diagonali-$\allowbreak$zation constant of $U$. Given such a choice of $N$, we can then consider $F = P_N U$ as the truncated forward map. The new sampling space for the reconstruction is given by $(\cD,\mu)=(\cB_N,\diff{x})$.

Recall that the natural forward map $U$ satisfies the quasi-diagonalization property with respect to $(\Phi,b)$ by definition. We can then invoke Proposition~\ref{prop:balancing} to conclude that $F$ satisfies the weak quasi-diagonalization property with respect to $(\Phi,j_0,b)$. Notice that, by Young's inequality,
\begin{align*}
    \|Fu\|_{L^2(\cD)} \leq \|Uu\|_{L^2} \leq \|\kappa_b\|_{L^1} \|u\|_{L^2}.
\end{align*}
We can therefore choose $C_F=\|\kappa_b\|_{L^1}$ in our setting.

\subsection{Coherence bounds and choice of the probability distribution}
\label{sec:deconv-coherencebounds-probdistr}
By Proposition~\ref{prop:app_lpnorm}, we have that
\begin{equation}
    \|(I-\Delta)^{-b/2}\phi_{j,n}\|_{\infty} \leq
    \tilde{B}_0 \frac{1}{2^{bj}} \|\phi_{j,n}\|_{\infty} \leq
    B_0 \frac{1}{2^{(b-1)j}},
\end{equation}
where $\tilde{B}_0$ and $B_0$ depend only  on the wavelet dictionary. This estimate can be written as
\begin{equation}
\label{eq:deconv_coherence}
    |F_t\phi_{j,n}|\leq B_0 \frac{1}{2^{(b-1)j}}.
\end{equation}
This bound can thus be recast in the form of \eqref{eq:coherence-natural} with $g_{\nu}=B_0^2$ and $d=b-1$. With reference to the normalization procedure outlined in Section \ref{sec:prob_distr}, we get that the previous estimate can be put in the form of Assumption~\ref{ass:coherence} with $B=B_0|\cB_N|^{1/2}$ and $f_{\nu}= 1/|\cB_N|$. Moreover, $f_{\nu}$ clearly satisfies Assumption \ref{ass:prob_lower_bound} with $c_{\nu}=1/|\cB_N|$.

Notice that, if $N=\lceil Cj_0 \rceil$, so that it satisfies the assumptions of Lemma~\ref{lem:deconv_balancing}, then
\begin{align*}
    |\cB_N| \asymp j_0^2.
\end{align*}
In particular, we get that
\begin{equation}
\label{eq:deconv-bcnu-est}
    B\asymp j_0,\qquad c_{\nu}\asymp j_0^{-2}.
\end{equation}

\subsection{Recovery estimates}
\label{sec:deconv-recovery}
Having verified all the assumptions, we can now directly apply Theorem \ref{thm:abstract} with $\zeta=1$ and with the uniform bound assumption on the noise level to deduce the following recovery estimates.

\begin{theorem}
\label{thm:deconv}
    Consider the wavelet dictionary $(\phi_{j,n})_{j,n\in\Gamma}$ defined in Sec.~\ref{sec:deconv-setting} and the corresponding analysis operator $\Phi\colon \cH_1\rightarrow\ell^2(\Gamma)$. Fix $j_0\in\bN$ and $b>2$.

    Consider a signal $u^{\dagger}\in L^2(\cB_1)$ satisfying $\|P_{\leq j_0}^{\perp} \Phi u^{\dagger}\|_2\leq r$ and let $t_1,\dots,t_m\in\cB_N$ be i.i.d.\ samples drawn from the uniform distribution on $\cB_N$, where $N=\lceil C j_0\rceil$ is chosen as in Lemma~\ref{lem:deconv_balancing}. Let $y_k \coloneqq (u^{\dagger}\ast\kappa_b)(t_k)+\varepsilon_k$ for $k=1,\dots,m$, with $|\varepsilon_k|\leq\beta$. Let $s\in\bN$ be such that $2\leq s\leq M_{\leq j_0}$, and set
    \begin{equation}
    \label{eq:tau-deconvolution}
        \tau\coloneqq j_0^2 2^{2j_0}s.
    \end{equation}
    
    There exist constants $C_0,C_1,C_2,C_3>0$, which depend only on $b$ and on the wavelet dictionary, for which the following result holds.
    
    Let $W=\diag(2^{bj})_{(j,n)\in\Lambda_{\leq j_0}}$ and let $\widehat{u}$ be a solution of the minimization problem
    \begin{align}
        \min_{u} \|W^{-1}\Phi{u}\|_1\quad\colon\quad \frac{1}{m} \sum_{k=1}^m |(u\ast\kappa_b)(t_k)-y_k|^2 \leq
        \lc\beta+ C_3 j_0 2^{-bj_0} r\rc^2,
    \end{align}
    where the minimum is taken over $\Span(\phi_{j,n})_{(j,n)\in\Lambda_{\leq j_0}}$.
    
    Let $\gamma\in(0,1)$. If
    \begin{equation}
    \label{eq:m-deconvolution}
        m\geq C_0\tau\max\{j_0\log^3{\tau},\log(1/\gamma)\},
    \end{equation}
    then, with probability exceeding $1-\gamma$, the following recovery estimate holds:
    \begin{equation}
        \|u^{\dagger}-\widehat{u}\|_{L^2} \leq
        C_1 2^{bj_0}\frac{\sigma_s(W^{-1}P_{\leq j_0}\Phi{u^{\dagger}})_1}{\sqrt{s}} + C_2 j_0 ( 2^{bj_0}\beta + j_0 r ).
    \end{equation}
\end{theorem}

\begin{remark}
We emphasize that the sample complexity given by \eqref{eq:tau-deconvolution} and \eqref{eq:m-deconvolution} does not lead to a classical subsampling result, since (up to logarithmic factors) it is proportional to $2^{2j_0}$ --- which is roughly the size of $|\Lambda_{\leq j_0}|$. This is mainly due to the high coherence condition encoded by \eqref{eq:deconv_coherence}. In particular, Theorem~\ref{thm:deconv} does not provide a subsampling strategy for the deconvolution problem, rather it allows one to achieve stable recovery with a number of measurements roughly proportional to $M_{\leq j_0} s$.

We also note that the ill-posedness of the convolution operator $F$ might imply that a number of samples larger than $M_{\leq j_0}$ is needed in order to obtain a method that is robust to noise and truncation error. This is the case, for instance, of function approximation problems, where it is known that $m\approx M^2$ pointwise equidistributed (with respect to the uniform measure) samples are required to stably approximate a linear combination of the first $M$ Legendre polynomials --- see \cite[Remark 2.2]{adcockinfdim} and the references \cite{RW,adcock_sparsebook,adcock2022efficient} on sparse polynomial approximation. Our result shows that a random sampling strategy is actually enough to obtain stable recovery with $m\approx M_{\leq j_0}s$, up to logarithmic terms. Moreover, the dependence of sample complexity on the degree of ill-posedness of the problem is confined to the constant $C_0$ in \eqref{eq:m-deconvolution}, while the dependence on $j_0$ is always proportional to $M_{\leq j_0}$ (up to logarithmic terms).

Similar remarks apply  to Theorems~\ref{thm:deconv2} and \ref{thm:inverse-source} as well.
\end{remark}

\subsection{An alternative probability distribution}
\label{sec:deconv-alternative-prob}
Sampling with respect to the uniform probability is not the only possible choice. Indeed, our general framework makes it possible to exploit the fast decay at infinity of the kernel to formulate a recovery procedure which allows us to sample on $\bR^2$ with respect to an exponentially decreasing probability density by slightly increasing the number of samples required, provided that a stronger assumption on the noise level is satisfied.

\begin{lemma}
\label{lem:deconv_alternative_coherence}
There exist constants $B\geq 1$ and $C>0$, depending only on $b$ and on the wavelet dictionary, such that, for every $\alpha\in(0,1]$, the following bound holds:
\begin{align*}
    |F_t \phi_{j,n}| \leq B\alpha^{-1} \frac{\sqrt{f_{\nu}(t)}}{2^{dj}},\qquad t\in\bR^2,\, (j,n)\in\Gamma
\end{align*}
with $d=(1-\alpha)(b-1)$, where $f_{\nu}$ is the probability density function proportional to $e^{-C\alpha|\cdot|}$ with respect to the Lebesgue measure on $\bR^2$.
\end{lemma}

We use this estimate for $\alpha=\eta/(b-1)$ for some $\eta\in(0,b-1]$ to apply Theorem \ref{thm:abstract} with $\zeta=1$ and with the nonuniform bound assumption on the noise level.
\begin{theorem}
\label{thm:deconv2}
    Consider the wavelet dictionary $(\phi_{j,n})_{j,n\in\Gamma}$ defined in Sec.~\ref{sec:deconv-setting} and the corresponding analysis operator $\Phi\colon \cH_1\rightarrow\ell^2(\Gamma)$. Fix $j_0\in\bN$, $b>2$ and $\eta\in(0,b-1]$.
    
    Consider a signal $u^{\dagger}\in L^2(\cB_1)$ satisfying $\|P_{\leq j_0}^{\perp}u^{\dagger}\|_{\cH_1}\leq \eta r$ for some $r\geq 0$ and let $t_1,\dots,t_m\in\bR^2$ be i.i.d.\ samples drawn from the distribution with density $f_{\nu}$ proportional to $e^{-C\eta|\cdot|/(b-1)}$, where the constant $C$ comes from Lemma~\ref{lem:deconv_alternative_coherence}. Let $y_k \coloneqq (u^{\dagger}\ast\kappa_b)(t_k)+\varepsilon_k$ for $k=1,\dots,m$, with $|\varepsilon_k|\leq \lc f_{\nu}(t_k)\rc^{1/2}\beta$. Let $s\in\bN$ be such that $2\leq s\leq M_{\leq j_0}$, and set
      \begin{align*}
        \tau \coloneqq \eta^{-2}2^{2 (1+\eta) j_0}s.
    \end{align*}
    
    There exist constants $C_0,C_1,C_2,C_3>0$, which depend only on $b$ and on the wavelet dictionary, for which the following result holds.
    
    Let $W=\diag(2^{bj})_{(j,n)\in\Lambda_{\leq j_0}}$ and let $\widehat{u}$ be a solution of the minimization problem
    \begin{align}
        \min_{u} \|W^{-1}\Phi{u}\|_1\quad\colon\quad \frac{1}{m} \sum_{k=1}^m f_{\nu}(t_k)^{-1}|(u\ast\kappa_b)(t_k)-y_k|^2 \leq
        \lc \beta+C_3 2^{-bj_0}r \rc^2,
    \end{align}
    where the minimum is taken over $\Span(\phi_{j,n})_{(j,n)\in\Lambda_{\leq j_0}}$.

    Let $\gamma\in(0,1)$. If
    \begin{align*}
        m\geq C_0\tau\max\{j_0\log^3{\tau},\log(1/\gamma)\},
    \end{align*}
    then, with probability exceeding $1-\gamma$, the following recovery estimate holds:
    \begin{equation}
        \|u^{\dagger}-\widehat{u}\|_{L^2} \leq
        C_1 2^{bj_0}\frac{\sigma_s(W^{-1}P_{\leq j_0}\Phi{u^{\dagger}})_1}{\sqrt{s}} + C_2 ( 2^{bj_0}\beta + r ).
    \end{equation}
\end{theorem}

\subsection{Deconvolution of cartoon-like images}
\label{sec:deconv_cartoon_like}
We now specialize our previous result to the deconvolution problem for cartoon-like images, which are, in short, $C^2$ signals apart from $C^2$ edges --- see \cite[Section 9.2.4]{Ma} for a precise definition.

Let $u^{\dagger}$ be a cartoon-like image supported in $\cB_1$. There exist some constants $C_0,C_1>0$ depending only on $b$, on the wavelet basis and on the parameters defining the cartoon-like images class such that, if $j_0\coloneqq\lfloor 2/(1+2b)\log(1/\beta) \rfloor$ and $m$ is sufficiently large\footnote{See the proof below for a quantitative version of this statement.}, the bound of Theorem \ref{thm:deconv} has the following explicit form: 
\begin{align}
\label{eq:ex-cartoon-like-est}
    \|u^{\dagger}-\widehat{u}\|_{L^2} \leq C_1 \lc \frac{\beta^{-2\frac{b+1}{2b+1}}\log^{3}(1/\beta)}{m^{1/2}}+\log^2(1/\beta) \beta^{\frac{1}{2b+1}} \rc.
\end{align}
Choosing $m=\lceil C_0 \beta^{ -2\frac{2b+3}{2b+1}}\log^2(1/\beta) \rceil$, up to logarithmic terms we have
\begin{align*}
    \|u^{\dagger}-\widehat{u}\|_{L^2} \leq C_1\beta^{\frac{1}{2b+1}},\qquad
    \|u^{\dagger}-\widehat{u}\|_{L^2} \leq C_1\lc \frac{1}{m} \rc^{\frac{1}{4b+6}}.
\end{align*}

\section{ Sparse inverse source problem for an elliptic PDE}
\label{sec:elliptic}
We consider an inverse source problem for an elliptic PDE \cite{biros-dogan-2008}, namely inverting the solution operator of a boundary value problem associated to a second order elliptic PDE. For simplicity, in this work we consider internal data of the solution; in the literature, the case with boundary data has been considered as well \cite{elbadia-2000}. This is a simple example of an inverse problem that is not modelled via a translation-equivariant forward map.

\subsection{Setting}
\label{sec:elliptic-setting}
We consider a $r$-regular ($r>2$) wavelet dictionary $(\phi_{j,n})_{(j,n)}$ adapted to the unit ball $\cB_1\subset\bR^2$ and the corresponding space $\cH_1$, as described in Sec.~\ref{sec:wavelet-uniform}. Let $\Omega\subseteq\bR^2$ be a $C^2$ bounded domain with $\Omega\Supset K$ --- see \eqref{eq-K} --- and $\sigma\in C^1(\overline{\Omega})$ be a coefficient with $\|\sigma\|_{C^1(\overline{\Omega})}\leq C_{\sigma}$ and $\sigma\geq\lambda>0$ in $\overline{\Omega}$.

We consider the following boundary value problem: 
\begin{equation}
\label{eq:ell_probl}
    \begin{cases}
        -\Div(\sigma\nabla{w}) = u & \text{in $\Omega$},\\
        w|_{\partial\Omega} = 0 & \text{on $\partial\Omega$},
    \end{cases}
\end{equation}
where $u\in L^2(\Omega)$ is a source and $w\in H^1_0(\Omega)$ is the unique solution \cite[Section 6.2]{Ev}.
The forward map $F\colon\cH_1\rightarrow L^2(\Omega)$ of the problem of interest is given by the solution operator $Fu=w$, where $w$ is the solution of \eqref{eq:ell_probl} with source $u$. We are interested in the reconstruction of the source $u$ from $w(t_k)$, $k=1,\dots,m$.

The measurement operators $F_t\colon \cH_1\rightarrow\bC$, $t \in \Omega$, are defined by $F_t{u} = F{u}(t)$ (here $\cH_2=\bC$), and the sampling space associated to the problem is given by $(\cD,\mu) = (\Omega,\diff{x})$, where $\diff{x}$ is the Lebesgue measure restricted to $\Omega$. Standard elliptic regularity theory guarantees that $F{u}\in H^2(\Omega)$ and that the following estimate holds:
\begin{equation}
\label{eq:h2reg}
    \|F{u}\|_{H^2(\Omega)} \leq C\|u\|_{L^2},
\end{equation}
where $C=C(\Omega,C_{\sigma},\lambda)$ --- see, for instance, \cite[Theorem 8.12]{gilbargtrudinger}. By Sobolev embedding \cite[Theorem 7.17]{gilbargtrudinger}, we have that, in dimension $2$, $H^2(\Omega)\subset C(\overline{\Omega})$ is a continuous inclusion. The two estimates together imply that $(F_t)_{t\in\Omega}$ is a well-defined family of uniformly bounded operators.

\subsection{Quasi-diagonalization and coherence bounds}
We first state some Sobolev stability estimates for the inverse problem.
\begin{lemma}
\label{lem:elliptic-sobolev}
Consider the setting introduced in Sec.~\ref{sec:elliptic-setting}. There exist constants $C_1=C_1(C_{\sigma})>0$ and $C_2=C_2(\Omega,K,C_{\sigma},\lambda)>0$ such that
\begin{align*}
    C_1\|u\|_{H^{-2}(\Omega)} \leq \|F{u}\|_{L^2(\Omega)} \leq C_2 \|u\|_{H^{-2}(\Omega)},\qquad u\in\cH_1.
\end{align*}
\end{lemma}
From Proposition~\ref{prop:littlewood-paley}, we deduce that $F$ satisfies the quasi-diagonalization property with $b=2$.

We also have the following coherence bounds.
\begin{lemma}
\label{lem:elliptic-coherence}
Consider the setting introduced in Sec.~\ref{sec:elliptic-setting}. There exists a constant $B_0>0$ (depending only on $\Omega$, $C_{\sigma}$, $\lambda$ and on the wavelet dictionary) such that, for every $\alpha\in(0,1]$,
    \begin{align*}
        |F_t{\phi_{j,n}}| \leq B_0 \alpha^{-1/2} \frac{1}{2^{(1-\alpha)j}},\quad t\in\Omega,\, (j,n)\in\Gamma.
    \end{align*}
\end{lemma}
By the normalization procedure outlined in Sec.~\ref{sec:prob_distr}, we conclude that Assumption~\ref{ass:coherence} is satisfied with $B=B_0|\Omega|^{1/2}\alpha^{-1/2}$, $d=1-\alpha$ and $f_{\nu}= 1/|\Omega|$, namely the density of the uniform distribution on $\Omega$.

\subsection{Recovery estimates}
We now apply Theorem \ref{thm:abstract} with $\zeta=1$ and with the uniform bound assumption on the noise level.
\begin{theorem}
\label{thm:inverse-source}
    Consider the setting introduced in Sec.~\ref{sec:elliptic-setting}. Fix $j_0\in\bN$ and $\alpha\in(0,1]$.
    Consider a signal $u^{\dagger}\in L^2(\cB_1)$ satisfying $\|P_{\leq j_0}^{\perp} \Phi u^{\dagger}\|_2\leq r$ and let $t_1,\dots,t_m\in\Omega$ be i.i.d.\ samples drawn from the uniform distribution on $\Omega$. Let $y_k \coloneqq w(t_k)+\varepsilon_k$ for $k=1,\dots,m$, with $|\varepsilon_k|\leq\beta$, where $w=F{u^{\dagger}}$ is the solution of \eqref{eq:ell_probl}. Let $s\in\bN$ be such that $2\leq s\leq M_{\leq j_0}$, and set
    \begin{align*}
        \tau\coloneqq \alpha^{-1}2^{2(1+\alpha)j_0}s.
    \end{align*}
    
    There exist constants $C_0,C_1,C_2,C_3>0$, which depend only on $\Omega$, $C_{\sigma}$, $\lambda$ and on the wavelet dictionary, such that the following holds.
    
    Let $W=\diag(2^{2j})_{(j,n)\in\Lambda_{\leq j_0}}$ and let $\widehat{u}$ be a solution of the minimization problem
    \begin{align}
        \min_{u} \|W^{-1}\Phi{u}\|_1\quad\colon\quad \frac{1}{m} \sum_{k=1}^m |F{u}(t_k)-y_k|^2 \leq
        \lc \beta+C_3  2^{-2j_0}r \rc^2,
    \end{align}
    where the minimum is taken over $\Span(\phi_{j,n})_{(j,n)\in\Lambda_{\leq j_0}}$.

    Let $\gamma\in(0,1)$. If
    \begin{align*}
        m\geq C_0\tau\max\{j_0\log^3{\tau},\log(1/\gamma)\},
    \end{align*}
    then, with probability exceeding $1-\gamma$, the following recovery estimate holds:
    \begin{equation}
        \|u^{\dagger}-\widehat{u}\|_{L^2} \leq
        C_1 2^{2j_0}\frac{\sigma_s(W^{-1}P_{\leq j_0}\Phi{u^{\dagger}})_1}{\sqrt{s}} + C_2 ( 2^{2j_0}\beta + r ).
    \end{equation}
\end{theorem}

\section{Fourier subsampling with ill-posedness}\label{sec-fou-ill}

In Example \ref{ex:fou-wav} we introduced the problem of reconstructing a signal $u^{\dagger}\in L^2(0,1)$ from some samples of its Fourier coefficients $\Big(\widehat{u^{\dagger}}(t)\Big)_{t\in \bZ}$. In this section we introduce a two-dimensional ill-posed version of this problem that is motivated by Magnetic Resonance Imaging (MRI). Indeed, in MRI one aims at reconstructing a signal $u^{\dagger}\in L^2(\cB_1)$ (recall that $\cB_1=\{x \in \bR^2 : |x| <1\}$) via the following samples along curves $k\colon[0,T]\rightarrow \bR^2$ for a discrete set of times $t_1,\dots,t_m$ --- see, for instance, \cite{nishimura}:
\begin{align*}
    s(t) = \int_{\bR^2} e^{-t/T_2(x)} u^{\dagger}(x) e^{-2\pi i k(t)\cdot x} \diff{x},
\end{align*}
where $T_2\colon\cB_1\rightarrow(0,+\infty)$ is the so-called \textit{spin-spin relaxation time}.

Assuming $T_2$ to be constant, we get that
\begin{align*}
    s(t) = e^{-t/T_2} \int_{\bR^2} u^{\dagger}(x) e^{-2\pi i k(t)\cdot x} \diff{x} = e^{-t/T_2} \cF u^{\dagger}\big(k(t)\big).
\end{align*}
In many cases, $k$ is chosen as a curve with $k(0)=0$ such that $|k|$ is increasing with time --- common choices are radial lines or outward spirals. If this is the case, we can see from the formula above that the actual samples are given by the Fourier transform of a signal multiplied by a function that decays to $0$ when $|k|\rightarrow+\infty$. This motivates the following simplified model.

\subsection{Setting}
\label{sec:illfour-setting}
We consider a $r$-regular ($r>1$) wavelet dictionary $(\phi_{j,n})_{(j,n)}$ adapted to the unit ball $\cB_1\subset\bR^2$ and the corresponding space $\cH_1$, as described in Sec.~\ref{sec:wavelet-uniform}. 
We also consider an open ball centered in the origin $\cB_R$ such that $\cB_R\Supset K$ --- see \eqref{eq-K}. We also set $d= d(\partial \cB_R,K)$ and introduce some related notions.
\begin{definition}[\cite{beurling}]
Let $Z\subset\bR^2$. We say that $Z$ is \textit{$\delta$-dense} if
\begin{align*}
    \delta = \sup_{y\in\bR^2} d(y,Z).
\end{align*}
We say that $Z$ is \textit{$\eta$-separated} if
\begin{align*}
    \inf_{ \substack{t_1,t_2\in Z \\ t_1\neq t_2} }
    |t_1-t_2| \geq \eta.
\end{align*}
\end{definition}
In our setting, we consider $Z\subset\bR^2$ which is $\delta$-dense for some $\delta\in(0,\frac{1}{4R})$ and $\eta$-separated for some $\eta>0$; for simplicity, we also assume that $0\in Z$. Under these assumptions, the dictionary $(\tilde{\psi}_t)_{t\in Z}$, defined by
\begin{align*}
    \tilde{\psi}_t(x) \coloneqq \exp(2\pi i t\cdot x)\mathbbm{1}_{\cB_R},
\end{align*}
is a frame of $L^2(\cB_R)$, with frame bounds depending only on $R$, $\delta$ and $\eta$ --- see \cite[Theorem 1]{beurling} for further details.

We consider a moderately ill-posed problem by modulating the signal via a Bessel operator: fix $b \in (0,r)$ and consider
\begin{align*}
    \psi_t \coloneqq (1+|t|^2)^{-b/2} \tilde{\psi}_t,\qquad t\in Z.
\end{align*}
We are interested in reconstructing a signal $u^{\dagger}\in L^2(\cB_1)$ from the knowledge of $\langle u^{\dagger},\psi_t \rangle$ for a finite number of $t\in Z$. The natural space for this problem is then $(\Sigma,\mu)=(Z,\mu)$, where $\mu$ is the counting measure on $Z$. The measurement operators $F_t\colon\cH_1\rightarrow\bC$ (here $\cH_2=\bC$) and the natural forward map $U\colon\cH_1\rightarrow\ell^2(Z)$ are given by
\begin{align*}
    F_t{u}=\langle u,\psi_t \rangle,\qquad U{u}(t)=\lc\langle u,\psi_t \rangle\rc_{t\in Z}.
\end{align*}
We explicitly note that $U$ is well defined because $(\tilde{\psi}_t)_{t\in Z}$ is a frame of $L^2(\cB_R)$:
\[
\|Uu\|_2^2 = \sum_{t\in Z} |\langle u,\psi_t \rangle|^2 = \sum_{t\in Z} (1+|t|^2)^{-b} |\langle u,\tilde \psi_t \rangle|^2 \le \sum_{t\in Z} |\langle u,\tilde \psi_t \rangle|^2 \lesssim \|u\|_{\cH_1}^2.
\]

\subsection{Balancing property and quasi-diagonalization} 
We first investigate the quasi-diagonalization property of the natural forward map $U$.
\begin{proposition}
\label{prop:illfour-quasidiag}
There exist constants $c,C>0$, depending only on $b$, on the frame bounds of $(\tilde{\psi}_t)_{t\in Z}$, on $K$ and on $R$, such that the following holds:
\begin{align*}
    c\|u\|_{H^{-b}(\bR^2)} \leq \|Uu\|_{\ell^2(Z)} \leq C\|u\|_{H^{-b}(\bR^2)},\qquad u\in L^2(K).
\end{align*}
\end{proposition}

Therefore the natural forward map $U$ satisfies the quasi-diagonalization property with respect to $(\Phi,b)$ by Proposition~\ref{prop:littlewood-paley}.
We now investigate how to perform truncation in order to preserve the energy of the forward map: we consider the projections $(P_{N})_{N\in\bN}$ on the spaces $\ell^2(Z\cap\cB_{N})$, so that $\cD=Z\cap \cB_{N}$.
\begin{lemma}
\label{lem:illfour-balancing}
There exists a constant $C\geq 1$, depending only on the upper frame bound of $(\tilde{\psi}_t)_{t\in Z}$, such that, if
\begin{align*}
    N \geq \theta^{-1/b} C^{1/2b} 2^{j_0},
\end{align*}
then $P_N$ satisfies the balancing property with respect to $(U,\Phi,j_0,b,\theta)$.
\end{lemma}
By Proposition \ref{prop:balancing} we get that $F=P_N\circ U$ satisfies the weak quasi-diagonalization property with respect to $(\Phi,j_0,b)$. We can fix $\theta=\sqrt{c/2}$, where $c$ is the lower quasi-diagonalization constant of $U$.

We emphasize that the constant $C^{1/b}$ in the Lemma blows up as $b\rightarrow 0^+$. We do not expect these estimates to be sharp: indeed, as already remarked in Example~\ref{ex:fouwave-continues}, it is possible to prove (under suitable assumptions on the wavelet dictionary) that a balancing property holds in the case $b=0$ for $N\gtrsim 2^{j_0}$.

\subsection{Coherence bounds and choice of the probability distribution}
\label{subsec:illfour-coherence}
\begin{lemma}
\label{lem:illfour-coherence}
There exists a constant $B_0>0$, depending only on $\eta$ and on the wavelet dictionary, such that
\begin{equation}
\label{eq:fou_wav_illposed_coh}
    |F_t\phi_{j,n}| \leq B_0 \frac{1}{2^{bj}(1+|t|^2)^{1/2}},\quad t\in Z,\ (j,n)\in\Gamma.
\end{equation}
\end{lemma}
Following the normalization procedure outlined in Section~\ref{sec:prob_distr}, we choose a probability density $f_\nu$ on $\cD=Z\cap\cB_N$ (with respect to the counting measure) proportional to $(1+|\cdot|^2)^{-1}$.
\begin{lemma}
\label{lem:illfour-zestimate}
Given $C_0>0$, there exists a constant $C_1$, depending only on $C_0$ and on $\eta$, such that the following holds:
\begin{align*}
    \#\{ t\in Z\colon\, |t|\in[0,C_0) \} &\leq C_1,\\
    \#\{ t\in Z\colon\, |t|\in[C_0 2^j,C_0 2^{j+1}) \} &\leq C_1 2^{2j},\qquad j\in\bN.
\end{align*}
\end{lemma}
    We now recast the inequality of Lemma~\ref{lem:illfour-coherence} in the form of Assumption~\ref{ass:coherence}. Choosing $N=\lceil C_0 2^{j_0}\rceil$, where $C_0=\theta^{-1/b}C^{1/b}$ as in Lemma~\ref{lem:illfour-balancing}, we consider the corresponding constant $C_1$ given by Lemma~\ref{lem:illfour-zestimate}. The normalization constant of the probability density is given by
\begin{align*}
    C_{\nu} \coloneqq \sum_{t\in\cD} \frac{1}{1+|t|^2} &=
    \sum_{\substack{ t\in Z \\ |t|\in[0,C_0) }} \frac{1}{1+|t|^2} +
    \sum_{j=0}^{j_0-1} \sum_{\substack{ t\in Z \\ |t|\in[C_0 2^j,C_0 2^{j+1}) }} \frac{1}{1+|t|^2} \\ &\leq
    C_1 + \sum_{j=0}^{j_0-1} \sum_{\substack{ t\in Z \\ |t|\in[C_0 2^j,C_0 2^{j+1}) }} \frac{1}{1+C_0^2 2^{2j}} \\ &\leq
    C_1 + \sum_{j=0}^{j_0-1} \frac{C_1 2^{2j}}{1+C_0^2 2^{2j}} \leq C_1 (j_0+1).
\end{align*}
We deduce that, for $N=C_0 2^{j_0}$, there exists $B\geq 1$, depending only on $\eta$, on the wavelet dictionary and on the upper frame bound, such that
\begin{align*}
    |F_t \phi_{j,n}| \leq B\sqrt{j_0} \frac{\sqrt{f_{\nu}(t)}}{2^{bj}},\qquad t\in Z\cap\cB_N,\ (j,n)\in\Gamma.
\end{align*}

\subsection{Recovery estimates}
We can now apply Theorem \ref{thm:abstract} with $\zeta=1$ and with the nonuniform bound assumption on the noise level.
\begin{theorem} \label{thm:four-nonu}
    Consider the lattice $Z$ and the wavelet dictionary $(\phi_{j,n})_{j,n\in\Gamma}$ defined in Sec.~\ref{sec:illfour-setting} and the corresponding analysis operator $\Phi\colon \cH_1\rightarrow\ell^2(\Gamma)$. Fix $j_0\in\bN$ and $b>0$.

    Consider a signal $u^{\dagger}\in L^2(\cB_1)$ satisfying \eqref{eq:ass-truncation-error2} for some $r\geq 0$ and let $t_1,\dots,t_m\in Z\cap\cB_N$, where $N=\lceil C^{1/b} 2^{j_0}\rceil$ as in Lemma~\ref{lem:illfour-balancing}, be i.i.d.\ samples drawn from the distribution with density $f_{\nu}$ proportional to $(1+|\cdot|^2)^{-1}$ with respect to the counting measure on $Z\cap\cB_N$. Let $y_k \coloneqq (1+|t_k|^2)^{-b/2}\widehat{u^{\dagger}}(t_k)+\varepsilon_k$ for $k=1,\dots,m$, with $|\varepsilon_k|\leq \lc f_{\nu}(t_k)\rc^{1/2}\beta$. Let $s\in\bN$ be such that $2\leq s\leq M_{\leq j_0}$.
    
    There exist constants $C_0,C_1,C_2,C_3>0$, which depend only on $\eta$, $\delta$, $R$ and on the wavelet dictionary, for which the following result holds.
    
    Let $W=\diag(2^{bj})_{(j,n)\in\Lambda_{\leq j_0}}$ and let $\widehat{u}$ be a solution of the minimization problem
    \begin{align}
        \min_{u} \|W^{-1}\Phi{u}\|_1\quad\colon\quad \frac{1}{m} \sum_{k=1}^m f_{\nu}(t_k)^{-1}|(1+|t_k|^2)^{-b/2}\widehat{u}(t_k)-y_k|^2 \leq
        \lc \beta+C_3 2^{-bj_0}r \rc^2,
    \end{align}
    where the minimum is taken over $\Span(\phi_{j,n})_{(j,n)\in\Lambda_{\leq j_0}}$.

    Let $\gamma\in(0,1)$. If
    \begin{align*}
        m\geq C_0 j_0 s\max\{j_0\log^3{(j_0 s)},\log(1/\gamma)\},
    \end{align*}
    then, with probability exceeding $1-\gamma$, the following recovery estimate holds:
    \begin{equation}
        \|u^{\dagger}-\widehat{u}\|_{L^2} \leq
        C_1 2^{bj_0}\frac{\sigma_s(W^{-1}P_{\Lambda_{\leq j_0}}\Phi{u^{\dagger}})_1}{\sqrt{s}} + C_2 ( 2^{bj_0}\beta + r ).
    \end{equation}
\end{theorem}

\begin{remark}
\label{rem:illfour-nonunif-trunc}
Note that the nonuniform bound \eqref{eq:ass-truncation-error2} on the truncation error has the following form: 
\begin{align*}
    \sup_{t\in Z} C j_0^{1/2} (1+|t|^2)^{(1-b)/2} |\widehat{u_R}(t)| \leq r.
\end{align*}
where $u_R= \Phi^* P_{\leq j_0}^{\perp}\Phi u^{\dagger}$. It is natural to investigate conditions such that
\begin{align*}
    \sup_{t\in Z} (1+|t|^2)^{(1-b)/2} |\widehat{u_R}(t)| \leq r
\end{align*}
holds, neglecting the factor $j_0$ for simplicity. If $b\geq 1$, then
\begin{align*}
    \sup_{t\in Z} (1+|t|^2)^{(1-b)/2} |\widehat{u_R}(t)| \leq
    \Big( \sum_{t\in Z}|\widehat{u_R}(t)|^2 \Big)^{1/2} \lesssim \|u_R\|_{\cH_1}
\end{align*}
by the frame property. Therefore, if $\|u_R\|_{\cH_1}\leq r$, then condition \eqref{eq:ass-truncation-error2} holds with respect to the same $r$, up to a constant.

If $b\in(0,1)$ instead, we get that
\begin{align*}
    \sup_{t\in Z} (1+|t|^2)^{(1-b)/2} |\widehat{u_R}(t)| \leq
    \Big( \sum_{t\in Z} (1+|t|^2)^{1-b}|\widehat{u_R}(t)|^2 \Big)^{1/2}.
\end{align*}
By way of arguments similar to the ones presented in the proof of Proposition~\ref{prop:illfour-quasidiag}, we infer 
\begin{align*}
    \Big( \sum_{t\in Z} (1+|t|^2)^{1-b}|\widehat{u_R}(t)|^2 \Big)^{1/2} \lesssim \|u_R\|_{H^{1-b}(B_R)}.
\end{align*}
In particular, if $\|u_R\|_{H^{1-b}(B_R)}\leq r$, then condition \eqref{eq:ass-truncation-error2} holds (up to a constant) with respect to the same $r$.
\end{remark}

\section{Proofs}\label{sec-proofs}

\subsection{Proofs of the abstract results}
\begin{proof}[Proof of Theorem \ref{thm:abstract}]
The result in the case of the uniform bound on the noise \eqref{eq:noise_unif_bound} corresponds to \cite[Theorem 3.11]{split1}, up to minor modifications. Indeed, notice that in the present version of the Theorem we have dropped the assumption \begin{align*}
    \sup_{t\in\cD} \|F_t\|\leq C_F,
\end{align*}
which was convenient in the context of \cite{split1}, and replaced it by
\begin{align*}
    \sup_{t\in\cD} \|F_t \Phi^* P_{\leq j_0}^{\perp} x^{\dagger}\|_{\cH_2}\leq r.
\end{align*}
As remarked in \cite[Section 2.2]{split1}, the uniform bound on $\|F_t\|$ affects only the truncation error; a careful inspection of \cite[Proposition 5.8]{split1} shows that the estimates hold also with the new assumption -- see also \cite[Remark 5.9]{split1}.

To obtain the result in the case of the non-uniform bound on the noise \eqref{eq:noise_nonunif_bound}, we consider the measurement operators given by $F_t'\coloneqq f_{\nu}(t)^{-1/2}F_t$. The nonuniform bound on the noise \eqref{eq:noise_nonunif_bound} for the measurements $F_t$ is thus equivalent to a uniform bound on the noise for the measurements $F_t'$. Condition \eqref{eq:ass-truncation-error} still holds after replacing $F_t$ with $F_t'$. We define the corresponding forward map as $F'u(t)\coloneqq F_t'{u}$ and set $\diff{\mu}'\coloneqq f_{\nu}\diff{\mu}$. Then, 
\begin{align*}
    \|F'u\|_{L_{\mu'}(\cD;\cH_2)}^2 = \int_{\cD} |f_{\nu}(t)|^{-1} \|F_t{u}\|_{\cH_2}^2 \diff{\mu}'(t) = \int_{\cD} \|F_t{u}\|_{\cH_2}^2 \diff{\mu}(t)=\|Fu\|_{L_{\mu}(\cD;\cH_2)}^2.
\end{align*}
In particular, $F'$ is well defined and Assumption \ref{ass:quasi-diag} is still satisfied. Moreover, the coherence bounds in Assumption \ref{ass:coherence} are still satisfied by the measurements $F_t'$ with respect to the probability density $f_{\nu}'= 1$. Finally, the probability density $f_{\nu}'$ trivially satisfies Assumption \ref{ass:prob_lower_bound} with respect to $c_{\nu}'=1$.

Applying the result in the case of the uniform bound to the measurements $F_t'$, the result follows.
\end{proof}

\begin{proof}[Proof of Proposition \ref{prop:balancing}]
Clearly, the quasi-diagonalization of $U$ and the fact that $P$ is a projection imply that
\begin{align*}
    \|F\Phi^*x\|_{L_{\mu}^2(\cD;\cH_2)}^2 =
    \|PU\Phi^*x\|_{\cH}^2 \leq
    \|U\Phi^* x\|_{\cH}^2 \leq
    C \sum_{(j,n)\in\Gamma} 2^{-2bj}|x_{j,n}|^2.
\end{align*}
This proves the first inequality of the weak quasi-diagonalization property of $F$.

In order to prove the second inequality, we consider $x\in\ell^2(\Lambda_{\leq j_0})$ and exploit the balancing property and the quasi-diagonalization of $U$ as follows:
\begin{align*}
    \|F\Phi^*x\|_{L_{\mu}^2(\cD;\cH_2)}^2 &=
    \|U\Phi^*x\|_{\cH}^2 -
    \|P^{\perp}U\Phi^*x\|_{\cH}^2 \\ &\geq
    c\sum_{(j,n)\in\Lambda_{\leq j_0}} 2^{-2bj} |x_{j,n}|^2 - \frac{c}{2} 2^{-2bj_0} \|x\|^2 \\ &=
    c\sum_{(j,n)\in\Lambda_{\leq j_0}} 2^{-2bj} |x_{j,n}|^2 - \frac{c}{2} \sum_{(j,n)\in\Lambda_{\leq j_0}} 2^{-2bj_0}|x_{j,n}|^2 \\ &\geq
    \frac{c}{2}\sum_{(j,n)\in\Lambda_{\leq j_0}} 2^{-2bj} |x_{j,n}|^2.
\end{align*}
This proves the second inequality and therefore the Proposition.
\end{proof}

\subsection{Proofs for sparse deconvolution}

\begin{proof}[Proof of Lemma~\ref{lem:deconv_balancing}]
The Bessel kernel $\kappa_b$ can be explicitly characterized --- see, for instance, \cite{AS}. We have that $\kappa_b\in C(\bR^2)$ and 
\begin{equation}
\label{eq:kernel_decay}
    |\kappa_b(x)| \leq C_1 e^{-C_2|x|},
\end{equation}
for constants $C_1,C_2>0$ that depend on $b$. Thus,  for $u\in \cH_1$  (hence compactly supported in $K$), we have that
\begin{align*}
    |u\ast\kappa_b(x)| &\leq
    \int_{K} |\phi(y)| |\kappa_b(x-y)|\diff{y} \\ &\leq
    C_1 \int_{K} |u(y)|e^{-C_2|x-y|}\diff{y} \\ &\leq
    C_1 e^{-C_2 |x|} \int_{K} |u(y)|e^{C_2|y|}\diff{y} \\ &\leq
    C_1 e^{-C_2 |x|} \lc \int_{K} e^{2C_2|y|}\diff{y} \rc^{1/2} \|u\|_{\cH_1}. 
\end{align*}
Therefore,
\begin{equation}
    \label{eq:decay_conv}
    |u\ast\kappa_b(x)|\leq
    C_1' e^{-C_2|x|} \|u\|_{\cH_1},\qquad x\in\bR^2,
\end{equation}
where the constant $C_1'$ depends only on $b$ and $K$. In particular, 
\[
\|P_N^{\perp}U u\|_{L^2(\bR^2)}^2 =\int_{\cB_N^c} |u\ast\kappa_b(x)|^2\,\diff{x} \le {C'_1}^2 \int_{\cB_N^c} e^{-2C_2|x|} \,\diff{x} \|u\|_{\cH_1}^2
\le {C''_1}^2  e^{-2C_2'N} \|u\|_{\cH_1}^2
\]
where $C_1',C_2'$  depend only on $K$ and $b$. In other words, $\|P_N^{\perp} U\Phi^*\iota_{\le j_0}\|_{\cH_1\to L^2(\bR^2)}^2 \leq C_1'' e^{-C_2' N}$.

Recall that $P_N$ satisfies the balancing property with respect to $(U,\Phi,j_0,b,\theta)$ if
\[
\|P_N^{\perp} U\Phi^*\iota_{\le j_0}\|_{\cH_1\to L^2(\bR^2)}^2 \leq \theta^2 2^{-2bj_0}.
\]
Thus, it is sufficient to show that
\[
 C_1'' e^{-C_2'N} \leq \theta^2 2^{-2bj_0},
\]
which follows from the bound
\begin{align*}
    N \geq C_3 \lc j_0 + \log(1/\theta) + 1 \rc,
\end{align*}
where $C_3>0$ depends on $C_1',C_2'$ and $b$.
\end{proof}

\begin{proof}[Proof of Theorem~\ref{thm:deconv}]
Proposition~\ref{prop:balancing} applied to \eqref{eq:deconv-quasi-diag} and Lemma~\ref{lem:deconv_balancing} imply that $F$ satisfies the weak quasi-diagonalization property with respect to $(\Phi,j_0,b)$, provided that $N$ is chosen large enough. Moreover, as outlined in Section~\ref{sec:deconv-coherencebounds-probdistr}, the estimate \eqref{eq:deconv_coherence} can be put in the form of Assumption~\ref{ass:coherence} with $B$ and $c_{\nu}$ as in \eqref{eq:deconv-bcnu-est}.

The claim then follows after applying Theorem~\ref{thm:abstract} in the case of uniform bounds on the noise and the truncation error.
\end{proof}

\begin{proof}[Proof of Lemma~\ref{lem:deconv_alternative_coherence}]
In view of \eqref{eq:decay_conv} applied with $u=\phi_{j,n}$, we have the following coherence bound for the wavelet dictionary
$(\phi_{j,n})_{(j,n)\in\Gamma}$:
\begin{align*}
    |F_t{\phi_{j,n}}| \leq C_1' e^{-C_2|t|}.
\end{align*}
On the other hand, recall from \eqref{eq:deconv_coherence} the following coherence bounds for some constant $B_0$ depending only on the wavelet dictionary:
\begin{equation}
    |F_t\phi_{j,n}|\leq B_0 \frac{1}{2^{(b-1)j}}.
\end{equation}
Now let $\alpha\in(0,1]$. We conclude that
\begin{equation}
\label{eq:deconv_new_coherence}
    |F_t{\phi_{j,n}}| \leq B_1 \frac{e^{-C_2\alpha|t|}}{2^{(1-\alpha)(b-1)j}},
\end{equation}
where $B_1=\max(C_1',B_1)$. 

Let 
\begin{align*}
    C \coloneqq \int_{\bR^2} e^{-2 C_2|x|}\,\diff{x},
\end{align*}
which clearly depends only on $C_2$. By a change of variable, we obtain $C\alpha^{-2}=\|e^{-2 C_2\alpha|\cdot|}\|_{L^1}$. We can then recast \eqref{eq:deconv_new_coherence} in the form of Assumption~\ref{ass:coherence} as
\begin{align*}
    |F_t{\phi_{j,n}}| \leq B\alpha^{-1} \frac{\sqrt{f_{\nu}(t)}}{2^{dj}},
\end{align*}
with $B=B_1 C^{1/2}$, $d=(1-\alpha)(b-1)$ and $f_{\nu}$ is the probability density proportional to $e^{-2C_2\alpha|\cdot|}$ with respect to the Lebesgue measure on $\bR^2$.
\end{proof}

\begin{proof}[Proof of Theorem~\ref{thm:deconv2}] 
We verify the conditions required to apply Theorem~\ref{thm:abstract} with nonuniform 
bounds on the noise and truncation error.

First, the quasi-diagonalization property of $F$ with respect to $(\Phi,b)$ follows 
directly from \eqref{eq:deconv-quasi-diag}. Second, Lemma~\ref{lem:deconv_alternative_coherence} 
establishes that Assumption~\ref{ass:coherence} holds with $d=(1-\alpha)(b-1)$.

For the nonuniform bound \eqref{eq:ass-truncation-error2}, we need to verify:
\begin{align*}
    \sup_{t\in\cD} f_{\nu}(t)^{-1/2} \|F_t \Phi^* P_{\leq j_0}^{\perp}x^{\dagger}\|_{\cH_2}\leq r.
\end{align*}
From \eqref{eq:decay_conv}, we know that for every $u\in\cH_1$,
\begin{align*}
    |u\ast\kappa_b(x)|\leq C e^{-C_2|x|} \|u\|_{\cH_1},\qquad x\in\bR^2,
\end{align*}
where $C_2$ is from \eqref{eq:decay_conv} and $C \geq 1$ depends only on $b$ and $K$. 
Since $f_{\nu}$ is proportional to $e^{-2C_2\alpha|\cdot|}$ with $\alpha=\eta/(b-1)$, 
we obtain
\begin{align*}
    \sup_{t\in\bR^2} f_{\nu}(t)^{-1/2} |P_{\leq j_0}^{\perp}u^{\dagger}\ast\kappa_b(x)| 
    \leq C\alpha^{-1} \|P_{\leq j_0}^{\perp}u^{\dagger}\|_{\cH_1} \leq C (b-1) r,
\end{align*}
where the last inequality uses our hypothesis $\|P_{\leq j_0}^{\perp}u^{\dagger}\|_{\cH_1}\leq \eta r$.

Therefore, condition \eqref{eq:ass-truncation-error2} holds with $C(b-1)r$ in place of $r$. 
The conclusion then follows from Theorem~\ref{thm:abstract}.
\end{proof}

\begin{proof}[Proof of the estimates in Section \ref{sec:deconv_cartoon_like}]
We first introduce the convenient notation \[ \tilde{m}\coloneqq m/(C_0 j_0\log^3{\tau}),\] where $C_0$ is the constant appearing in Theorem \ref{thm:deconv} and $\tau\coloneqq j_0^2 2^{2j_0}s$ for some $s\geq 2$. Notice that, by the choice of $j_0$, we have $j_0\asymp \log(1/\beta)$ and 
\begin{align*}
    \log{\tau} \asymp \log{j_0}+j_0+\log{s} \asymp j_0+\log{M_{\leq j_0}}\asymp j_0 \asymp \log(1/\beta).
\end{align*}
We conclude that
\begin{align}
\label{eq:rel-m-mtilde}
    \tilde{m} \asymp m/\log^4(1/\beta).
\end{align}
Suppose that $\beta$ is sufficiently small so that $(\log^3{\tau})\, j_0\geq \log(1/\gamma)$ --- recall that $j_0\asymp\log(1/\beta)$.
By Theorem \ref{thm:deconv}, if
\begin{equation}
\label{eq:tildem_sparsity}
    \tilde{m} = \lceil j_0^2 2^{2j_0} s \rceil,
\end{equation}
then, with probability exceeding $1-\gamma$, the following error estimate holds:
\begin{equation}
    \|u^{\dagger}-\widehat{u}\|_{L^2} \leq
    C_1 2^{bj_0}\frac{\sigma_s(W^{-1}P_{\leq j_0}\Phi{u^{\dagger}})_1}{\sqrt{s}} + C_2 j_0 ( 2^{bj_0}\beta + j_0 r ).
\end{equation}
Take now $s_1,\dots,s_{j_0}\in\bN$ such that $s_1+\dots+s_{j_0}=s$. We have
\[
\begin{split}
   \sigma_s(W^{-1}P_{\leq j_0}\Phi{u^{\dagger}})_1 
   &= \inf\{\| (W^{-1}P_{\leq j_0}\Phi{u^{\dagger}})_{\Lambda_{\leq j_0}\setminus S}\|_1 : |S|\leq s\}
   \\ &\leq \inf\left\{\sum_{j=1}^{j_0}\| (P_j W^{-1}P_{\leq j_0}\Phi{u^{\dagger}})_{\Lambda_{j}\setminus S_j}\|_1 : S_j\subseteq \Lambda_j,\ |S_j|\leq s_j\right\}
   \\ &=\sum_{j=1}^{j_0}   2^{-bj}\inf\left\{\| (P_j \Phi{u^{\dagger}})_{\Lambda_{j}\setminus S_j}\|_1 : S_j\subseteq \Lambda_j,\ |S_j|\leq s_j\right\}
   \\ &=\sum_{j=1}^{j_0}   2^{-bj}\sigma_{s_j}(P_j \Phi{u^{\dagger}})_1.
\end{split}
\]
Therefore
\begin{equation}
    \label{eq:sparsity_est}
    \|u^{\dagger}-\widehat{u}\|_{L^2} \leq \frac{C_1}{\sqrt{s}} 2^{bj_0} \sum_{j=1}^{j_0} 2^{-bj} \sigma_{s_j}(P_{j}\Phi u^{\dagger})_1 + C_2 j_0 ( 2^{bj_0}\beta + j_0 r ).
\end{equation}

Following the arguments used in \cite[Section 9.3.1]{Ma}, we have that $\|\phi_{j,n}\|_{L^2}=1$ and $|\supp(\phi_{j,n})|\lesssim C2^{-2j}$, and therefore
\begin{align*}
    |\langle u^{\dagger},\phi_{j,n} \rangle| \leq \|u^{\dagger}\|_{L^{\infty}} \|\phi_{j,n}\|_{L^1} \lesssim 2^{-j}.
\end{align*}
On the other hand, in view of the Littlewood-Paley property of wavelets with $b=-2$ --- see Proposition~\ref{prop:littlewood-paley} in Appendix~\ref{appendix:wavelets} --- we have that, if $\supp(\phi_{j,n})$ does not intersect the discontinuities of $u^\dagger$, then, choosing a smooth cutoff $\eta$ supported in the vicinity of $\supp(\phi_{j,n})$,
\begin{align*}
    |\langle u^{\dagger},\phi_{j,n} \rangle| = |\langle u^{\dagger}\eta,\phi_{j,n} \rangle| \leq
    \|u^{\dagger}\eta\|_{H^2} \|\phi_{j,n}\|_{H^{-2}}\lesssim 2^{-2j}.
\end{align*}
Let $\Sigma_j^1$ be the set of indices of wavelets at scale $j$ whose support intersects the discontinuities of $u^\dagger$, and  $\Sigma_j^2$ be the set of indices of wavelets at scale $j$ whose support does not intersects the discontinuities of $u^\dagger$. It is possible to see, arguing as in \cite[Section 9.3.1 --- \textit{NonLinear Approximation of Piecewise Regular Images}]{Ma}, that $|\Sigma_j^1|\asymp 2^j$ and $|\Sigma_j^2|\asymp 2^{2j}$. We can then compute the $\ell^1$-norm of $P_{j}\Phi u^{\dagger}$ for every $j$:
\begin{align*}
    \|P_{j}\Phi u^{\dagger}\|_1 &=
    \sum_{(j,n)\in\Sigma_j^1} |\langle u^{\dagger},\phi_{j,n} \rangle| + \sum_{(j,n)\in\Sigma_j^2} |\langle u^{\dagger},\phi_{j,n} \rangle| \\ &\lesssim
    2^j 2^{-j} + 2^{2j} 2^{-2j} = 2.
\end{align*}
We conclude that $\sigma_{s_j}(P_{j}\Phi u^{\dagger})_1 \leq \|P_{j}\Phi{u^{\dagger}}\|_1 \lesssim 1$. A similar argument gives $r\lesssim 2^{-j_0/2}$. The implicit constants appearing in the bounds for the non-linear approximation error depend both on the wavelet basis and on parameters defining the cartoon-like images class.

Plugging these estimates into \eqref{eq:sparsity_est}, we get
\begin{align*}
    \|u^{\dagger}-\widehat{u}\|_{L^2} 
    &\lesssim
    \frac{1}{\sqrt{s}} 2^{bj_0} \sum_{j\leq j_0} 2^{-bj} + j_0 ( 2^{bj_0}\beta + j_0 2^{-j_0/2} ) \\ &\lesssim
    \frac{2^{bj_0}}{\sqrt{s}} + j_0^2 ( 2^{bj_0}\beta + 2^{-j_0/2} ).
\end{align*}
Imposing that $2^{bj_0}\beta \asymp 2^{-j_0/2}$ we get $2^{bj_0}\asymp\beta^{-b/(b+1/2)}$, and therefore
\begin{align*}
    \|u^{\dagger}-\widehat{u}\|_{L^2} \lesssim \frac{\beta^{-2b/(2b+1)}}{\sqrt{s}} + \log^2(1/\beta) \beta^{1/(2b+1)}.
\end{align*}
Using relation \eqref{eq:tildem_sparsity}, we get
\begin{align*}
    \|u^{\dagger}-\widehat{u}\|_{L^2} \lesssim \frac{\beta^{-2\frac{b+1}{2b+1}}\log(1/\beta)}{\tilde{m}^{1/2}}+\log^2(1/\beta) \beta^{\frac{1}{2b+1}}.
\end{align*}
Moreover, if $\tilde{m}= \beta^{ -2\frac{2b+3}{2b+1} }\log^{-2}(1/\beta)$, then
\begin{align*}
    \|u^{\dagger}-\widehat{u}\|_{L^2} \lesssim \log^2(1/\beta) \beta^{1/(2b+1)}.
\end{align*}
The sample complexity for $m$ follows from \eqref{eq:rel-m-mtilde}.
\end{proof}

\subsection{Proofs for the sparse inverse source problem}
The proofs of the results in this section resort to standard techniques from elliptic PDE theory --- we address the interested reader to classical references such as \cite{Ev, gilbargtrudinger, necas, lions}. Nevertheless, since we were not able to locate in the literature the exact estimates that are needed for our purposes, we decided to report detailed proofs for the sake of completeness.
\begin{proof}[Proof of Lemma~\ref{lem:elliptic-sobolev}]
We first prove that
\begin{equation}
\label{eq:bound1-sobolev}
    C_1\|u\|_{H^{-2}(\Omega)} \leq \|Fu\|_{L^2(\Omega)}.
\end{equation}
Indeed, let $u\in\cH_1$ and let $v\in H_0^2(\Omega)$. We have that
\begin{align*}
    \langle u,v \rangle &=
    \int_{\Omega} \sigma\nabla(Fu)\cdot\nabla{v} \\ &=
    \int_{\partial\Omega} Fu(\sigma\nabla{v}\cdot\nu) -
    \int_{\Omega} Fu\, \Div(\sigma\nabla{v}) \\ &=
    -\int_{\Omega} Fu\, (\nabla{\sigma}\cdot\nabla{v}+\sigma\Delta{v}) \\ &\leq
    2C_{\sigma}\|v\|_{H_0^2(\Omega)} \|Fu\|_{L^2(\Omega)}.
\end{align*}
Taking the supremum over $v\in H_0^2(\Omega)$ with $\|v\|_{H_0^2(\Omega)}\leq 1$ implies \eqref{eq:bound1-sobolev}.

We now prove that
\begin{equation}
\label{eq:bound2-sobolev}
    \|F{u}\|_{L^2(\Omega)} \leq C_2 \|u\|_{H^{-2}(\Omega)}.
\end{equation}
Indeed, let $u\in\cH_1$ and let $v\in L^2(\Omega)$. By standard elliptic PDE theory \cite[Theorem 7.25]{gilbargtrudinger}, there exists a bounded extension operator $E\colon H^2(\Omega)\rightarrow H^2(\bR^2)$ such that
\begin{equation}
\label{eq:extension_op}
    Ew|_{\Omega} \equiv w,\qquad \|Ew\|_{H^2(\bR^2)}\leq C\|w\|_{H^2(\Omega)},
\end{equation}
where $C=C(\Omega)$. We have that
\begin{align*}
    \langle Fu,v \rangle &=
    \langle Fu,-\Div(\sigma\nabla{Fv}) \rangle \\ &=
    \int_{\Omega} \sigma\nabla(Fu)\cdot\nabla{Fv} \\ &=
    \int_{\Omega} u\, Fv \\ &=
    \int_{\bR^2} u\, EFv \\ &\leq
    \|u\|_{H^{-2}(\bR^2)} \|EFv\|_{H^2(\bR^2)} \\ &\leq
    C\|u\|_{H^{-2}(\bR^2)} \|Fv\|_{H^2(\Omega)} \\ &\leq
    C\|u\|_{H^{-2}(\bR^2)} \|v\|_{L^2(\Omega)},
\end{align*}
where we have used \eqref{eq:h2reg} in the last inequality; here $C$ depends on the norm of the extension operator and on the stability constant in \eqref{eq:h2reg}. Taking the supremum over $v\in L^2(\Omega)$ with $\|v\|_{L^2(\Omega)}\leq 1$, we conclude that
\begin{align*}
    \|Fu\|_{L^2(\Omega)} \leq C\|u\|_{H^{-2}(\bR^2)}.
\end{align*}
In order to prove \eqref{eq:bound2-sobolev}, we now show that, for all $u\in L^2(\bR^2)$ with $\supp(u)\subset K$,
\begin{align*}
    \|u\|_{H^{-2}(\bR^2)} \leq C\|u\|_{H^{-2}(\Omega)}
\end{align*}
for some constant $C>0$. Indeed, let $v\in H^2(\bR^2)$. Consider a smooth bump function $\eta\in C^{\infty}(\bR^2)$ such that
\begin{align*}
    \eta|_K \equiv 1,\qquad \supp(\eta)\subset\Omega.
\end{align*}
It follows that $\eta v\in H_0^2(\Omega)$. Moreover, we have that
\begin{align*}
    \langle u,v \rangle &= 
    \langle u,\eta v \rangle \leq
    \|u\|_{H^{-2}(\Omega)} \|\eta v\|_{H_0^2(\Omega)} \leq
    C \|u\|_{H^{-2}(\Omega)} \|v\|_{H^2(\bR^2)},
\end{align*}
where $C$ depends on $\eta$, and therefore only on $K$ and $\Omega$. Taking the supremum over all $v\in H^2(\bR^2)$ with $\|v\|_{H^2(\bR^2)}\leq 1$, we conclude.
\end{proof}

\begin{proof}[Proof of Lemma~\ref{lem:elliptic-coherence}]

    By standard elliptic PDE theory, we have that a solution of \eqref{eq:ell_probl} exists for $u\in H^{-1}(\Omega)$ and that the following estimate holds:
    \begin{align*}
        \|F{u}\|_{H^1(\Omega)} \leq C_1\|u\|_{H^{-1}(\Omega)},
    \end{align*}
    where $C_1=C_1(\Omega,C_{\sigma},\lambda)$. By regularity theory \cite[Theorem 9.25]{brezis}, we have that a solution of \eqref{eq:ell_probl} for $u\in L^2(\Omega)$ is actually in $H^2(\Omega)$ and that
    \begin{align*}
        \|Fu\|_{H^2(\Omega)} \leq C_2\|u\|_{L^2(\Omega)},
    \end{align*}
    where $C_2=C_2(\Omega,C_{\sigma},\lambda)$.

    By standard  elliptic PDE theory \cite[Theorem 7.25]{gilbargtrudinger}, there exists a bounded extension operator $E\colon H^1(\Omega)\rightarrow H^1(\bR^2)$ such that $E(H^2(\Omega))\subset H^2(\bR^2)$ and
    \begin{equation}
    Ew|_{\Omega} \equiv w,\quad \|Ew\|_{H^1(\bR^2)}\leq C\|w\|_{H^1(\Omega)},\quad \|Ew\|_{H^2(\bR^2)}\leq C\|w\|_{H^2(\Omega)},
    \end{equation}
    where $C=C(\Omega)$. We conclude that
    \begin{align*}
        \|EF{u}\|_{H^1(\bR^2)} \leq C_1'\|u\|_{H^{-1}(\Omega)},\quad
        \|EF{u}\|_{H^2(\bR^2)} \leq C_2'\|u\|_{L^2(\Omega)},
    \end{align*}
    where the constants $C_1',C_2'$ have the same dependencies of $C_1,C_2$ on the setting of the problem. Clearly, we have that $\|u\|_{L^2(\Omega)}\leq\|u\|_{L^2(\bR^2)}$.
    
    The following bound holds for $u\in L^2(\bR^2)$:
    \begin{align*}
        \|u\|_{H^{-1}(\Omega)} \leq \|u\|_{H^{-1}(\bR^2)}.
    \end{align*}
    Indeed, we have that, if $v\in H_0^1(\Omega)$, then  $\|\tilde v\|_{H^1(\bR^2)} = \|v\|_{H_0^1(\Omega)}$, where $\tilde v\in H^1(\bR^2)$ is the extension by $0$ of $v$ to  $\bR^2$. Therefore
    \begin{align*}
        |\langle u,v \rangle_{H^{-1}(\Omega)\times H^1_0(\Omega)}| &=|\langle u,\tilde v \rangle_{H^{-1}(\bR^2)\times H^1_0(\bR^2)}| \\ &\leq \|u\|_{H^{-1}(\bR^2)} \|\tilde v\|_{H^1(\bR^2)} \\ &= \|u\|_{H^{-1}(\bR^2)} \|v\|_{H_0^1(\Omega)}.
    \end{align*}
    Taking the supremum over all $v\in H_0^1(\Omega)$ with $\|v\|_{H_0^1(\Omega)}\leq 1$ proves the claim.
    
    We conclude that, for $u\in L^2(\bR^2)$,
    \begin{align*}
        \|EF{u}\|_{H^1(\bR^2)} \leq C_1'\|u\|_{H^{-1}(\bR^2)},\quad
        \|EF{u}\|_{H^2(\bR^2)} \leq C_2'\|u\|_{L^2(\bR^2)}.
    \end{align*}
      In particular, if $\eta\in\cD(\bR^2)$ is a cutoff function with $\supp(\eta)\subset\Omega$, $\eta\in[0,1]$ and $\eta|_K\equiv 1$, we get that, for $u\in L^2(\bR^2)$,
    \begin{align*}
        \|EF(\eta u)\|_{H^{1}(\bR^2)}&\leq C_1'\|\eta u\|_{H^{-1}(\bR^2)} \leq C_1'' \|u\|_{H^{-1}(\bR^2)},\\
        \|EF(\eta u)\|_{H^2(\bR^2)} &\leq C_2' \|u\|_{L^2(\bR^2)}.
    \end{align*}
    where $C_1''$ have the same dependencies of $C_1'$ on the setting of the problem, with an additional dependency on $\eta$ and thus on $K$.
    
    Complex interpolation \cite[Theorem 4.1.2, Theorem 6.4.5]{bergh2012interpolation} applied to the couples $\lc H^{-1}(\bR^2),L^2(\bR^2)\rc$ and $\lc H^1(\bR^2),H^2(\bR^2) \rc$ yields, for every $\alpha\in(0,1)$ and for every $u\in H^{1+\alpha}(\bR^2)$ with $\supp(u)\subset K$,
    \begin{align*}
        \|EFu\|_{H^{1+\alpha}(\bR^2)} \leq C\|u\|_{H^{-1+\alpha}(\bR^2)},
    \end{align*}
    where the constant $C$ can be chosen as $C=\max(C_1'',C_2')$, for instance.

    By Sobolev embedding \cite[Theorem 7.17]{gilbargtrudinger}, we have that $H^{1+\alpha}(\bR^2)\subset C(\bR^2)$ is a continuous inclusion, where the operator norm of the inclusion can be bounded by a constant that scales with $\alpha$ as $\alpha^{-1/2}$. Indeed, we have that, for $u\in H^{1+\alpha}(\bR^2)$, 
    \begin{align*}
        \|u\|_{L^{\infty}(\bR^2)} &\leq
        \|\widehat{u}\|_{L^1} =
        \int_{\bR^2} (1+|\xi|^2)^{(1+\alpha)/2}|\widehat{u}(\xi)|
        \frac{1}{(1+|\xi|^2)^{(1+\alpha)/2}}\,\diff{\xi} \\ &\leq
        \lc \int_{\bR^2} \frac{1}{(1+|\xi|^2)^{(1+\alpha)}}\,\diff{\xi} \rc^{1/2} \|u\|_{H^{1+\alpha}(\bR^2)} \\ &=
        \sqrt{\pi} \lc \int_0^{+\infty} \frac{1}{(1+t)^{1+\alpha}}\,\diff{t} \rc^{1/2} \|u\|_{H^{1+\alpha}(\bR^2)} \\ &=
        \sqrt{\pi} \alpha^{-1/2} \|u\|_{H^{1+\alpha}(\bR^2)}.
    \end{align*}
    We obtain that, for $u\in L^2(\bR^2)$ with $\supp(u)\subset K$,
    \begin{align*}
        \|F{u}\|_{L^{\infty}(\Omega)} & \leq
        \|EF{u}\|_{L^{\infty}(\bR^2)} \\ &\leq
        C\alpha^{-1/2} \|EF{u}\|_{H^{1+\alpha}(\bR^2)} \\ &\leq
        C\alpha^{-1/2} \|u\|_{H^{-1+\alpha}(\bR^2)}.
    \end{align*}
    Using the Littlewood-Paley property of wavelets (see Prop.~\ref{prop:littlewood-paley}), we conclude that
    \begin{align*}
        |F_t{\phi_{j,n}}| &\leq \|F{\phi_{j,n}}\|_{L^{\infty}(\Omega)} \\ &\leq C\alpha^{-1/2} \|\phi_{j,n}\|_{H^{-1+\alpha}(\Omega)} \\ &\leq B\alpha^{-1/2}\frac{1}{2^{(1-\alpha)j}},
    \end{align*}
    as desired.
\end{proof}

\begin{proof}[Proof of Theorem~\ref{thm:inverse-source}] 
We verify the conditions required to apply Theorem~\ref{thm:abstract}. First, 
Lemma~\ref{lem:elliptic-sobolev} establishes that the forward operator satisfies 
the quasi-diagonalization property with parameter $b=2$ with respect to the analysis 
operator $\Phi$. Additionally, the necessary coherence bounds for the measurement 
operators are provided by Lemma~\ref{lem:elliptic-coherence}. With these key 
conditions satisfied, the recovery guarantees follow directly from 
Theorem~\ref{thm:abstract}.
\end{proof}

\subsection{Proofs for Fourier subsampling with ill-posedness}

We first introduce some notation.

Let $\Psi\colon L^2(\cB_R)\rightarrow\ell^2(Z)$ be the analysis operator of $(\tilde{\psi}_t)_{t\in Z}$. For $s\in\bR$, let $h^s(Z)$ be the space
\begin{align*}
    h^s(Z) \coloneqq \bigg\{(x_t)_{t\in Z}\in\bC^Z\colon\ 
    \sum_{t\in Z} (1+|t|^2)^s|x_t|^2<+\infty \bigg\}.
\end{align*}
We note that $h^s(Z)$ is a Hilbert space endowed with the inner product
\begin{align*}
    \langle x,y \rangle_{h^s} \coloneqq \sum_{t\in Z} (1+|t|^2)^s x_t \overline{y_t}.
\end{align*}
For $k\in\bN$, the norm $\|\cdot\|_{h^k}$ is equivalent to the one induced by the following inner product --- which, with an abuse of notation, we still denote with $\langle \cdot,\cdot \rangle_{h^k}$:
\begin{align*}
    \langle x,y \rangle_{h^k} \coloneqq \sum_{t\in Z}\sum_{|\alpha|\leq k} (2\pi)^{2|\alpha|} t^{2\alpha} x_t\overline{y_t}.
\end{align*}
Note that the implicit constants in the equivalence of norms depend only on $k$.

For $k\in\bN$ and $u\in L^2(\cB_R)$, we recall the definition of the inner product on $H^k(\cB_R)$, given by
\begin{align*}
    \langle u,v \rangle_{H^k} = \sum_{|\alpha|\leq k} \langle \partial^{\alpha}u,\partial^{\alpha}v \rangle.
\end{align*}
For $u\in H_0^k(\cB_R)$, the norm $\|u\|_{H^k(\cB_R)}$ can be identified with $\|\tilde{u}\|_{H^k(\bR^2)}$, where $\tilde{u}$ is the extension by $0$ of $u$ to $\bR^2$. Again, the implicit constants in the equivalence of norms depend only on $k$.

Finally, for $k\in\bN$ and $u\in L^2(\cB_R)$, we define the norm
\begin{align*}
    \|u\|_{\tilde{H}^{-k}(\cB_R)} \coloneqq \sup\{|\langle u,v\rangle_{L^2}|\colon\ v\in H^k(\cB_R),\ \|v\|_{H^k(\cB_R)}\leq 1\}.
\end{align*}
The proof of Proposition~\ref{prop:illfour-quasidiag} requires several lemmas, to be proved separately.
\begin{lemma}
\label{lem:illfour-quasidiag-sublem1}
Let $k\in\bN$. Then $\Psi\big(H^k(\cB_R)\big)\subset h^k(Z)$ and there exists constants $c_1,C_1>0$, depending only on the frame bounds, such that
\begin{align*}
    c_1\|u\|_{H^k(\cB_R)} \leq \|\Psi{u}\|_{h^k(Z)} \leq C_1\|u\|_{H^k(\cB_R)},\qquad u\in H^k(\cB_R).
\end{align*}
\end{lemma}
\begin{lemma}
\label{lem:illfour-quasidiag-sublem2}
Let $k\in\bN$. Then $\Psi^* \Psi\big(H^k(\cB_R)\big)=H^k(\cB_R)$ and there exists constants $c_2,C_2>0$, depending only on the frame bounds, such that
\begin{align*}
    c_2\|u\|_{H^k(\cB_R)} \leq \|\Psi^* \Psi{u}\|_{H^k(\cB_R)} \leq C_2\|u\|_{H^k(\cB_R)},\qquad u\in H^k(\cB_R).
\end{align*}
Moreover, we have that
\begin{align*}
    C_2^{-1}\|u\|_{\tilde{H}^{-k}(\cB_R)} \leq \|(\Psi^* \Psi)^{-1}{u}\|_{\tilde{H}^{-k}(\cB_R)} \leq c_2^{-1} \|u\|_{\tilde{H}^{-k}(\cB_R)},\qquad u\in L^2(\cB_R).
\end{align*}
\end{lemma}
\begin{lemma}
\label{lem:illfour-quasidiag-sublem3}
Let $k\in\bN$. Then there exist constants $c_3,C_3>0$, depending only on $k$ and on the frame bounds, such that
\begin{align*}
    \|\Psi{u}\|_{h^{-k}(Z)} \leq C_3\|u\|_{\tilde{H}^{-k}(\cB_R)},\qquad u\in L^2(\cB_R)
\end{align*}
and
\begin{align*}
    c_3\|(\Psi^*\Psi)^{-1}\Psi^* x\|_{\tilde{H}^{-k}(\cB_R)} \leq \|x\|_{h^{-k}(Z)},\qquad x\in\ell^2(Z).
\end{align*}
\end{lemma}
\begin{lemma}
\label{lem:illfour-quasidiag-sublem4}
Let $k\in\bN$. Then there exist constants $c_4,C_4>0$, depending only on $k$, on the frame bounds, on $K$ and on $R$, such that, for all $b\in[0,k]$,
\begin{align*}
    c_4\|u\|_{H^{-b}(\bR^2)}\leq\|\Psi{u}\|_{h^{-b}(Z)}\leq C_4\|u\|_{H^{-b}(\bR^2)},\qquad u\in L^2(K).
\end{align*}
\end{lemma}


\begin{proof}[Proof of Proposition~\ref{prop:illfour-quasidiag}] 
We establish the quasi-diagonalization property by showing an equivalence between norms. 
First, observe that for any $u \in \mathcal{H}_1$,
\begin{align*}
    \|\Psi u\|_{h^{-b}(Z)} &= \left(\sum_{t \in Z} (1+|t|^2)^{-b}|\langle u,\tilde \psi_t\rangle|^2\right)^{1/2} \\
    &= \left(\sum_{t \in Z}|\langle u,(1+|t|^2)^{-b/2} \tilde \psi_t\rangle|^2\right)^{1/2} \\
    &= \|U u\|_{\ell^2(Z)}.
\end{align*}
The result then follows directly from Lemma~\ref{lem:illfour-quasidiag-sublem4}.
\end{proof}

\begin{proof}[Proof of Lemma~\ref{lem:illfour-quasidiag-sublem1}]
For $u\in H^k(\cB_R)$ and $|\alpha|\leq k$, we have that
\begin{align*}
    \|\Psi(\partial^{\alpha}u)\|_{\ell^2(Z)}^2 = \sum_{t\in Z}|\langle \partial^{\alpha}u,\tilde{\psi}_t\rangle_{L^2}|^2 =
    \sum_{t\in Z}(2\pi)^{2|\alpha|}t^{2\alpha}|\langle u,\tilde{\psi}_t\rangle_{L^2}|^2.
\end{align*}
We conclude that
\begin{align*}
    \sum_{|\alpha|\leq k} \|\Psi(\partial^{\alpha}u)\|_{\ell^2(Z)}^2 = \|\Psi{u}\|_{h^k(Z)}^2.
\end{align*}
We now use the frame property of $u$ to conclude that
\begin{align*}
    \|u\|_{H^k(\cB_R)}^2 = \sum_{|\alpha|\leq k} \|\partial^{\alpha}u\|_{L^2(\cB_R)}^2 \asymp \sum_{|\alpha|\leq k} \|\Psi(\partial^{\alpha}u)\|_{\ell^2(Z)}^2 = \|\Psi{u}\|_{h^k(Z)}^2.
\end{align*}
where the implicit constants depend only on the frame bounds.
\end{proof}

\begin{proof}[Proof of Lemma~\ref{lem:illfour-quasidiag-sublem2}]
Let $\Psi_k\colon H^k(\cB_R)\rightarrow h^k(Z)$ be defined by $\Psi_k=\Psi|_{H^k(\cB_R)}$. We now compute $\Psi_k^*$. Given $u\in H^k(\cB_R)$ and $y\in h^k(Z)$, we have that
\begin{align*}
    \langle \Psi_k u,y \rangle_{h^k} &=
    \sum_{t\in Z} \sum_{|\alpha|\leq k} (2\pi)^{2|\alpha|} t^{2\alpha} \langle u,\tilde{\psi}_t \rangle_{L^2} \overline{y_t} \\ &=
    \sum_{t\in Z} \sum_{|\alpha|\leq k} (2\pi i)^{|\alpha|}t^{\alpha} \langle \partial^{\alpha}u,\tilde{\psi}_t \rangle_{L^2} \overline{y_t} \\ &=
    \sum_{|\alpha|\leq k}
    \bigg\langle \partial^{\alpha}u,\sum_{t\in Z}(-2\pi i)^{|\alpha|}t^{\alpha}y_t\tilde{\psi}_t \bigg\rangle_{L^2} \\ &=
    \sum_{|\alpha|\leq k} \bigg\langle \partial^{\alpha}u,\partial^{\alpha}\sum_{t\in Z}y_t \tilde{\psi}_t \bigg\rangle_{L^2} \\ &=
    \langle u,\sum_{t\in Z}y_t\tilde{\psi}_t \rangle_{H^k}.
\end{align*}
We conclude that
\begin{align*}
    \Psi_k^* y = \sum_{t\in Z} y_t \tilde{\psi}_t.
\end{align*}
In other words, $\Psi_k^*=\Psi^*|_{h^k(Z)}$, which implies that
\begin{equation}
\label{eq:illfour-quasidiag-sublem2-identity}
    \Psi_k^* \Psi_k u = \Psi^* \Psi u,\qquad u\in H^k(\cB_R).
\end{equation}
By Lemma~\ref{lem:illfour-quasidiag-sublem1}, $\Psi_k$ is an isomorphism on its image; this implies that $\Psi_k^*\Psi_k$ is an automorphism of $H^k(\cB_R)$; this proves the first claim.

We now prove the second claim. Let $z\in L^2(\cB_R)$ and let $v\in H^k(\cB_R)$. Then we have that 
\begin{align*}
    |\langle \Psi^*\Psi z,v \rangle_{L^2}| &\leq 
    |\langle z,\Psi^*\Psi v \rangle_{L^2}| \\ &\leq
    \|z\|_{\tilde{H}^{-k}(\cB_R)} \|\Psi^*\Psi v\|_{H^k(\cB_R)} \\ &\leq
    C_2 \|z\|_{\tilde{H}^{-k}(\cB_R)} \|v\|_{H^k(\cB_R)},
\end{align*}
which implies that
\begin{align*}
    \|\Psi^* \Psi z\|_{\tilde{H}^{-k}(\cB_R)} \leq C_2 \|z\|_{\tilde{H}^{-k}(\cB_R)}.
\end{align*}
Let $w\in H^k(\cB_R)$ be such that $v=\Psi^* \Psi w$. Then
\begin{align*}
    |\langle z,v \rangle_{L^2}| &=
    |\langle z,\Psi^*\Psi w \rangle_{L^2}| \\ &=
    |\langle \Psi^*\Psi z,w \rangle_{L^2}| \\ &\leq
    \|\Psi^*\Psi z\|_{\tilde{H}^{-k}(\cB_R)} \|w\|_{H^k(\cB_R)} \\ &\leq
    c_2^{-1} \|\Psi^*\Psi z\|_{\tilde{H}^{-k}(\cB_R)} \|v\|_{H^k(\cB_R)},
\end{align*}
which implies that
\begin{align*}
   \|z\|_{\tilde{H}^{-k}(\cB_R)} \leq c_2^{-1}\|\Psi^*\Psi z\|_{\tilde{H}^{-k}(\cB_R)}.
\end{align*}
Setting $z=\Psi^*\Psi u$ concludes the proof of the second claim.
\end{proof}

\begin{proof}[Proof of Lemma~\ref{lem:illfour-quasidiag-sublem3}]
We prove the first claim. Let $u\in L^2(\cB_R)$ and $y\in h^k(Z)$. Then we have that
\begin{align*}
    |\langle \Psi{u},y \rangle_{\ell^2}| = |\langle u,\Psi^*{y} \rangle_{L^2}|.
\end{align*}
We now recall the identity $\Psi_k^*=\Psi^*|_{h^k(Z)}$ from the proof of Lemma~\ref{lem:illfour-quasidiag-sublem2}. By Lemma~\ref{lem:illfour-quasidiag-sublem1} we obtain
\begin{align*}
    |\langle \Psi{u},y \rangle_{\ell^2}| &=
    |\langle u,\Psi_k^* y \rangle_{L^2}| \\ &\leq
    \|u\|_{\tilde{H}^{-k}(\cB_R)} \|\Psi_k^*\|_{h^k\to H^k} \|y\|_{h^k(Z)} \\ &=
    \|u\|_{\tilde{H}^{-k}(\cB_R)} \|\Psi_k\|_{H^k\to h^k} \|y\|_{h^k(Z)} \leq
    C_1 \|u\|_{\tilde{H}^{-k}(\cB_R)} \|y\|_{h^k(Z)}.
\end{align*}
This implies that
\begin{align*}
    \|\Psi{u}\|_{h^{-k}(Z)} \leq C_3 \|u\|_{\tilde{H}^{-k}(\cB_R)}
\end{align*}
for $C_3=C_1$.

We now prove the second claim. Let $x\in\ell^2(Z)$ and let $v\in H^k(\cB_R)$. By Lemma~\ref{lem:illfour-quasidiag-sublem2}, we have that
\begin{align*}
    \|(\Psi^*\Psi)^{-1}\Psi^* x\|_{\tilde{H}^{-k}(\cB_R)} \leq
    c_2^{-1} \|\Psi^* x\|_{\tilde{H}^{-k}(\cB_R)}.
\end{align*}
We now use Lemma~\ref{lem:illfour-quasidiag-sublem1} to conclude that
\begin{align*}
    |\langle \Psi^* x,v \rangle_{L^2}| &=
    |\langle x,\Psi{v} \rangle_{\ell^2}| \\ &\leq
    \|x\|_{h^{-k}(Z)} \|\Psi{v}\|_{h^k(Z)} \leq
    C_1' \|x\|_{h^{-k}(Z)} \|v\|_{H^k(\cB_R)}.
\end{align*}
This implies that
\begin{align*}
    c_3 \|(\Psi^*\Psi)^{-1}\Psi^* x\|_{\tilde{H}^{-k}(\cB_R)} \leq \|x\|_{h^{-k}(Z)}
\end{align*}
for $c_3\coloneqq c_2/C_1'$.
\end{proof}

\begin{proof}[Proof of Lemma~\ref{lem:illfour-quasidiag-sublem4}]
We consider a smooth cutoff function $\eta\in\cD_{\cB_R}(\bR^2)$ with $\eta\in[0,1]$ and $\eta|_{K}\equiv 1$. Notice that, for $v\in H^k(\bR^2)$, we have that $\eta v\in H^k(\cB_R)$ and
\begin{align*}
    \|\eta v\|_{H^k(\cB_R)} \leq C'\|v\|_{H^k(\bR^2)},
\end{align*}
where $C'$  depends only on $k$ and $\eta$. We conclude that, for $w\in L^2(\cB_R)$,
\begin{align}
\label{eq:cut_boh}
    \|\eta w\|_{H^{-k}(\bR^2)} \leq C' \|w\|_{\tilde{H}^{-k}(\cB_R)}.
\end{align}
Analogously, if $v\in H^k(\cB_R)$, we have that $\eta v\in H^k(\bR^2)$ and
\begin{align*}
    \|\eta v\|_{H^k(\bR^2)} \leq C''\|v\|_{H^k(\cB_R)},
\end{align*}
where $C''$ only depends on $k$ and on $\eta$. This implies that, for $w\in L^2(\bR^2)$,
\begin{align}\label{eq:cut_boh2}
    \|\eta w\|_{\tilde{H}^{-k}(\cB_R)} \leq C''\|w\|_{H^{-k}(\bR^2)}.
\end{align}

We now prove the first inequality. Let $0<c\leq C$ denote the frame bounds of $\{\tilde\psi_t\}_{t\in Z}$. By $ \|(\Psi^*\Psi)^{-1}\|_{L^2(\cB_R)\to L^2(\cB_R)} \leq c^{-1}$ and $\|\Psi^*\|_{\ell^2(Z)\to L^2(\cB_R)}\leq \sqrt{C}$, we know that
\begin{align*}
    \|\eta(\Psi^*\Psi)^{-1}\Psi^* x\|_{L^2(\bR^2)} \leq \|(\Psi^*\Psi)^{-1}\Psi^* x\|_{L^2(\cB_R)} \leq \frac{\sqrt{C}}{c} \|x\|_{\ell^2(Z)},\qquad x\in \ell^2(Z).
\end{align*}
Moreover, the following holds:
\begin{align*}
    \frac{c_3}{C'}\|\eta(\Psi^*\Psi)^{-1}\Psi^* x\|_{H^{-k}(\bR^2)} \leq \|x\|_{h^{-k}(Z)},\qquad x\in \ell^2(Z).
\end{align*}
Indeed, by \eqref{eq:cut_boh} and Lemma~\ref{lem:illfour-quasidiag-sublem3}, we have that
\begin{align*}
    \|\eta(\Psi^*\Psi)^{-1}\Psi^* x\|_{H^{-k}(\bR^2)} \leq
    C' \|(\Psi^*\Psi)^{-1}\Psi^* x\|_{H^{-k}(\bR^2)} \leq
    \frac{C'}{c_3} \leq \|x\|_{h^{-k}(Z)}.
\end{align*}
Applying the complex interpolation method \cite[Theorem 4.1.2, Theorem 6.4.5]{bergh2012interpolation} to the compatible couples $\big( h^{-k}(Z),\ell^2(Z) \big)$ and $\big( H^{-k}(\bR^2),L^2(\bR^2) \big)$, we conclude that
\begin{align*}
    c_4 \|\eta(\Psi^*\Psi)^{-1}\Psi^* x\|_{H^{-b}(\bR^2)} \leq \|x\|_{h^{-b}(Z)},\qquad x\in \ell^2(Z),
\end{align*}
where $c_4$ can be chosen as $c_4\coloneqq \min(c/\sqrt{C},c_3/C')$, for instance. If $u\in L^2(K)$, we can set $x=\Psi{u}$ and get
\begin{align*}
    c_4\|u\|_{H^{-b}(\bR^2)} = c_4\|\eta u\|_{H^{-b}(\bR^2)} \leq \|\Psi{u}\|_{h^{-b}(Z)}.
\end{align*}

We now prove the second inequality. We have that
\begin{align*}
    \|\Psi(\eta u)\|_{\ell^2(Z)} \leq \sqrt{C}\|\eta u\|_{L^2(\cB_R)} \leq \sqrt{C}\|u\|_{L^2(\bR^2)},\qquad u\in L^2(\bR^2).
\end{align*}
Moreover, by \eqref{eq:cut_boh2} and Lemma~\ref{lem:illfour-quasidiag-sublem3}, we have that
\begin{align*}
    \|\Psi(\eta u)\|_{h^{-k}(Z)} \leq C_3\|\eta u\|_{\tilde{H}^{-k}(\cB_R)}\leq C'' C_3\|u\|_{H^{-k}(\bR^2)},\qquad u\in L^2(\bR^2).
\end{align*}
Applying the complex interpolation method \cite[Theorem 4.1.2, Theorem 6.4.5]{bergh2012interpolation} to the compatible couples $\big( H^{-k}(\bR^2),L^2(\bR^2) \big)$ and $\big( h^{-k}(Z),\ell^2(Z) \big)$, we conclude that
\begin{align*}
    \|\Psi(\eta u)\|_{h^{-b}(Z)} \leq C_4\|u\|_{H^{-b}(\bR^2)},\qquad u\in L^2(\bR^2).
\end{align*}
where $C_4$ can be chosen as $C_4\coloneqq\max(\sqrt{C},C''C_3)$, for instance. We conclude the proof by observing that $\eta u = u$ if $u\in L^2(K)$.
\end{proof}

\begin{proof}[Proof of Lemma~\ref{lem:illfour-balancing}]
Let $u\in L^2(K)$. Then we have that:
\begin{align*}
    \|P_N^{\perp} U{u}\|_{\ell^2}^2 &= \sum_{ \substack{ t\in Z \\ |t|>N } } (1+|t|^2)^{-b} |\widehat{u}(t)|^2 \\ &\leq
    (1+N^2)^{-b} \sum_{t\in Z} |\widehat{u}(t)|^2 \\ &\leq
    C(1+N^2)^{-b} \|u\|_{L^2}^2
    \\ &\leq C N^{-2b} \|u\|_{L^2}^2,
\end{align*}
where the constant $C\geq 1$ is the upper frame bound of $(\tilde{\psi}_t)_{t\in Z}$. The balancing property then holds if
\begin{equation*} 
    N \geq \theta^{-1/b} C^{1/2b} 2^{j_0}. \qedhere
\end{equation*}
\end{proof}

\begin{proof}[Proof of Lemma~\ref{lem:illfour-coherence}]
In this proof we consider the dictionary of 2D wavelets $(\psi_{n})_n\cup(\phi_{j,n,\varepsilon})_{j,n,\varepsilon}$ --- see Appendix~\ref{appendix:wavelets}. The element $(\phi_{j,n})_{j,n}$ will denote one of the three types of wavelets $(\phi_{j,n,\varepsilon})_{j,n}$ for $\varepsilon\in\{0,1\}^2\setminus\{(0,0)\}$. The estimates for the low frequency part $(\psi_{n})_n$ are obtained analogously.

For $t=0$, the inequality is trivial as $F_0\phi_{j,n}=0$. Let then $t\in Z\setminus\{0\}$ and notice that, as $Z$ is $\eta$-separated and $0\in Z$, we have that $|t|\geq\eta$. Suppose first that $0\leq b\leq 1$, so that $0\leq 1-b\leq 1$. As the wavelets are compactly supported and at least $1$-regular, there exists a constant $C>0$ (depending only on the wavelet dictionary) such that
\begin{align*}
    |\widehat{\phi_{0,0}}| \leq \frac{C}{|\cdot|^{1-b}}.
\end{align*}
Suppose now that $b>1$. By property \eqref{eq:wav_vanishing_moments}, we have that
\begin{align*}
    0 = (-2\pi i)^{|\alpha|}\int x^{\alpha}\phi_{0,0} = \partial^{\alpha}\widehat{\phi_{0,0}}(0),\qquad |\alpha|\leq b-1.
\end{align*}
This implies that there exists $C>0$ depending only on the wavelet dictionary such that
\begin{align*}
    |\widehat{\phi_{0,0}}| \leq \frac{C}{|\cdot|^{1-b}}.
\end{align*}
In both cases, we conclude that
\begin{align*}
    |F_t\phi_{j,n}| &= \frac{1}{(1+|t|^2)^{b/2}}|\widehat{\phi_{j,n}}(t)| \\ &=
    \frac{1}{(1+|t|^2)^{b/2}} 2^{-j} |\widehat{\phi_{0,0}}(2^{-j} t)| \\ &\leq
    \frac{C}{(1+|t|^2)^{b/2}} 2^{-j} \frac{1}{2^{-(1-b)j} |t|^{1-b} }  \\ &\leq
    \frac{C'}{2^{bj} (1+|t|^2)^{1/2}}
\end{align*}
for some constant $C'>0$ depending only on $\eta$ and on the wavelet dictionary. This proves \eqref{eq:fou_wav_illposed_coh}.
\end{proof}

\begin{proof}[Proof of Lemma~\ref{lem:illfour-zestimate}]
We only prove the second estimate, the first one being analogous. We fix $j\in\bN$ and consider the following sets:
\begin{align*}
    A_0 &\coloneqq \{ t\in\bR^2\colon\, |t|\in[C_0 2^j-\eta/2,C_0 2^{j+1}+\eta/2) \},\\
    A_1 &\coloneqq \bigcup_{\substack{t\in Z \\ |t|\in[C_0 2^j,C_0 2^{j+1})}} B_{\eta/2}(t).
\end{align*}
As $Z$ is $\eta$-separated, the balls in the definition of $A_1$ are pairwise disjoint. Therefore, we get that
\begin{align*}
    |A_1| = \#\{ t\in Z\colon\, |t|\in[C_0 2^j,C_0 2^{j+1}) \} \omega_2 \frac{\eta^2}{4},
\end{align*}
where $\omega_2\coloneqq |\cB_1|$. Moreover, by definition $A_1\subset A_0$, therefore
\begin{align*}
    |A_1| \leq |A_0|.
\end{align*}
We have
\begin{align*}
    |A_0| &= \omega_2 \lc (C_0 2^{j+1}+\eta/2)^2 - (C_0 2^j-\eta/2)^2 \rc \\ &=
    3\omega_2 (C_0^2 2^{2j}+C_0 2^j\eta).
\end{align*}
Putting everything together, we get that
\begin{align*}
    \#\{ t\in Z\colon\, |t|\in[C_0 2^j,C_0 2^{j+1}) \} \leq
    12 \frac{C_0^2 2^{2j}+C_0 2^j \eta}{\eta^2} \leq
    12 \frac{C_0^2+C_0\eta}{\eta^2} 2^{2j}.
\end{align*}
Choosing $C_1\coloneqq 12 (C_0^2+C_0\eta)/\eta^2$ yields the result.
\end{proof}



\begin{proof}[Proof of Theorem~\ref{thm:four-nonu}] 
We verify the conditions required to apply Theorem~\ref{thm:abstract} with nonuniform 
bounds on the noise and truncation error.

Let $N = \lceil C_0 2^{j_0} \rceil$. By Proposition~\ref{prop:illfour-quasidiag}, 
the natural forward map $U$ satisfies the quasi-diagonalization property. 
Furthermore, Lemma~\ref{lem:illfour-balancing} establishes that the projection 
$P_N$ satisfies the balancing property. Therefore, by Proposition~\ref{prop:balancing}, 
the truncated forward map $F = P_N \circ U$ satisfies the weak quasi-diagonalization 
property with respect to $(\Phi, j_0, b)$.

As shown in Section~\ref{subsec:illfour-coherence}, the coherence estimate from 
Lemma~\ref{lem:illfour-coherence} can be reformulated to satisfy 
Assumption~\ref{ass:coherence}.

With the above conditions verified, we can apply Theorem~\ref{thm:abstract} with 
nonuniform bounds on the noise and truncation error to obtain the desired recovery 
guarantees.
\end{proof}

\section*{Acknowledgments}

This material is based upon work supported by the Air Force Office of Scientific Research under award number FA8655-23-1-7083. Co-funded by the European Union (ERC, SAMPDE, 101041040). Views and opinions expressed are however those of the authors only and do not necessarily reflect those of the European Union or the European Research Council. Neither the European Union nor the granting authority can be held responsible for them. The authors are members of the ``Gruppo Nazionale per l’Analisi Matematica, la Probabilità e le loro Applicazioni'', of the ``Istituto Nazionale di Alta Matematica''. The research was supported in part by the MIUR Excellence Department Project awarded to Dipartimento di Matematica, Università di Genova, CUP D33C23001110001. Funded
by European Union – Next Generation EU.

\bibliography{refs}{}
\bibliographystyle{plain}

\appendix
\section{Wavelets and their properties}
\label{appendix:wavelets}
We briefly describe the framework of \textit{multiresolution approximation (MRA)} and the associated construction of wavelets --- we refer the reader to \cite[Chapter 2, 3]{Me} for further details.

\begin{definition}
A \textit{multiresolution approximation} of $L^2(\bR^\dim)$ is a sequence of subspaces $(V_j)_{j\in\bZ}\subset L^2(\bR^\dim)$ such that
\begin{enumerate}
    \item $V_j\subset V_{j+1}$ for all $j\in\bZ$,
    \item \begin{align*}
        \bigcap_{j\in\bZ} V_j = \emptyset,\qquad
        \overline{\bigcup_{j\in\bZ} V_j} = L^2(\bR^\dim),
    \end{align*}
    \item $f\in V_j$ if and only if $f(2\cdot)\in V_{j+1}$,
    \item there exists a function $g\in L^2(\bR^\dim)$ such that $(g(\cdot-n))_{n\in\bZ^\dim}$ is a Riesz basis of the space $V_0$, namely there exists constants $C_1,C_2>0$ such that, for every $x=(x_n)_{n\in\bZ^\dim}\subset\ell^2(\bZ^\dim)$,
    \begin{align*}
        C_1 \|x\|_{\ell^2}^2 \leq
        \bigg\| \sum_{n\in\bZ^d}x_n g(\cdot-n) \bigg\|_{L^2}^2 \leq C_2 \|x\|_{\ell^2}^2.
    \end{align*}
\end{enumerate}
\end{definition}
\begin{definition}
Let $r\in\bN$. A multiresolution approximation $(V_j)_{j\in\bZ}\subset L^2(\bR^\dim)$ is called \textit{$r$-regular} if the function $g$ in the definition can be chosen in such a way that $g\in C^r(\bR^\dim)$ and, for every multi-index $\alpha$ with $|\alpha|\leq r$ and for every $m\in\bN$, there exists $C_m>0$ such that
\begin{align*}
    |\partial^{\alpha}g(x)| \leq \frac{C_m}{(1+|x|^2)^{m/2}},\qquad x\in\bR^\dim.
\end{align*}
\end{definition}
\begin{proposition}[{\cite[Theorem 1, Section 2.4]{Me}}]
Let $(V_j)_{j\in\bZ}$ be a $r$-regular multiresolution approximation of $L^2(\bR^\dim)$ and let $g$ be the function in the definition. Consider the functions
\begin{align*}
    \widehat{\psi} \coloneqq \frac{\widehat{g}}{\lc \sum_{n\in\bZ^d} |\widehat{g}(\cdot+n)|^2 \rc^{1/2}},\qquad
    \psi_{0,n} \coloneqq \psi(\cdot-n).
\end{align*}
Then $(\psi_{0,n})_{n\in\bZ^\dim}$ defines an orthonormal basis of $V_0$ that is $r$-regular.
\end{proposition}
The function $\psi$ in the Proposition is called the \textit{scaling function} of the multiresolution approximation.

We now consider the one dimensional case.
\begin{proposition}[{\cite[Theorem 1, Section 3.2]{Me}}]
Let $(V_j)_{j\in\bZ}$ be a $r$-regular multiresolution approximation of $L^2(\bR)$. For $j\in\bZ$, let $W_j$ denote the orthogonal complement of $V_{j}$ in $V_{j+1}$. Then
\begin{align*}
    L^2(\bR) = \bigoplus_{j\in\bZ} W_j.
\end{align*}
Moreover, there exists a function $\phi\in W_0$ such that, defining
\begin{align*}
    \phi_{j,n}(x) \coloneqq 2^{j/2}\phi(2^jx-n),\qquad j\in\bZ,\,n\in\bZ,
\end{align*}
then $(\phi_{j,n})_{j,n}$ is an $r$-regular orthonormal basis of $L^2(\bR)$, with $\phi_{j,n}\in W_j$.
\end{proposition}
The following result is a fundamental tool for the construction of compactly supported wavelets in dimension $1$.
\begin{proposition}[{\cite[Theorem 3, Section 3.8]{Me}}]
Let $r\in\bN$. Then there exists a $r$-regular multiresolution approximation $(V_j)_{j\in\bZ}$ of $L^2(\bR)$ such that the associated functions $\psi$ and $\phi$ have compact support.
\end{proposition}

We now describe the construction of \textit{separable wavelet bases} for $L^2(\bR^2)$; the construction can be also generalized to higher dimension.
\begin{proposition}[{\cite[Section 7.7.1, 7.7.2]{Ma}, \cite[Section 3.9]{Me}}]
Let $(V_j)$ be a $r$-regular multiresolution approximation of $L^2(\bR)$ with associated functions $\psi$ and $\phi$. For $j\in\bN$, $n=(n_1,n_2)\in\bZ^2$ and $\varepsilon=(\varepsilon_1,\varepsilon_2)\in\{0,1\}^2\setminus\{(0,0)\}$, we define
\begin{align*}
    \psi_n^2(x_1,x_2) &\coloneqq \psi(x_1-n_1)\psi(x_2-n_2),\\
    \phi_{j,n,\varepsilon}^2(x_1,x_2) &\coloneqq 2^j \phi^{\varepsilon_1}(2^j x_1 - n_1)\phi^{\varepsilon_2}(2^j x_2-n_2),
\end{align*}
where $\phi^{0}=\psi$ and $\phi^{1}=\phi$. Then $(\psi_n^2)_{n\in\bZ}$ generates a $r$-regular multiresolution approximation $(V_j^2)_{j\in\bZ}$ of $L^2(\bR^2)$.

Moreover, let $W_j^2$ be the orthogonal complement of $V_j^2$ in $V_{j+1}^2$. Then $W_j^2$ is generated by $(\phi_{j,n,\varepsilon}^2)_{n,\varepsilon}$; therefore, the family $(\psi_n^2)_{n\in\bZ}\cup(\phi_{j,n,\varepsilon}^2)_{j,n,\varepsilon}$ is an orthonormal basis of $L^2(\bR^2)$. Finally, the following holds: 
\begin{equation}
\label{eq:wav_vanishing_moments}
    \int x^{\alpha}\phi_{j,n,\varepsilon} = 0,\quad
    |\alpha|\leq r.
\end{equation}
\end{proposition}
We now describe the \textit{Littlewood-Paley property} of $r$-regular multiresolution approximations.
\begin{proposition}[{\cite[Theorem 8, Section 2.8]{Me}}]
\label{prop:littlewood-paley}
Let $(V_j)_{j\in\bZ}$ be a $r$-regular multiresolution approximation of $L^2(\bR^\dim)$ and let $\psi$ and $\phi$ be the associated functions. Let $|s|<r$. Then there exists constants $C_1,C_2>0$ such that the following Littlewood-Paley property holds for $f\in H^s(\bR^\dim)$:
\begin{align*}
    C_1 \|f\|_{H^s}^2 \leq \|P_{V_0}f\|_{L^2}^2 + \sum_{j} 2^{2sj} \|P_{W_j}f\|_{L^2}^2 \leq C_2 \|f\|_{H^s}^2,
\end{align*}
where $P_{V_0}$ and $P_{W_j}$ denote the orthogonal projections on $V_0$ and $W_j$, respectively.
\end{proposition}
Finally, the following estimates for $L^p$ norms of Bessel kernels holds.
\begin{proposition}{\cite[Proposition 3, Section 2.8]{Me}}
\label{prop:app_lpnorm}
Let $(V_j)_{j\in\bZ}$ be a $r$-regular multiresolution approximation of $L^2(\bR^\dim)$. There exists constants $C_1,C_2>0$ such that the following holds: for $1\leq p\leq \infty$, $|s|\leq r$ and $j\in\bN$,
\begin{align*}
    C_1 2^{sj} \|f\|_{L^p} \leq \|(I-\Delta)^{s/2}f\|_{L^p} \leq C_2 2^{sj}\|f\|_{L^p},\qquad f\in W_j\cap L^p.
\end{align*}
\end{proposition}

\end{document}